\documentclass[draft]{amsart}
\usepackage{amssymb,amsfonts,amsmath,amsthm,cite}
\usepackage{enumerate}
\usepackage{ytableau}
\usepackage{graphicx, xcolor}
\usepackage{hyperref}
\hypersetup{colorlinks=true,linkcolor=blue,filecolor=magenta,urlcolor=cyan}
\numberwithin{equation}{section}

\allowdisplaybreaks[4]

\theoremstyle{plain}
\newtheorem{theorem}{Theorem}[section]
\newtheorem{lemma}[theorem]{Lemma}
\newtheorem{corollary}[theorem]{Corollary}
\newtheorem{definition}[theorem]{Definition}
\newtheorem{conjecture}[theorem]{Conjecture}

\title{A MacMahon Analysis View of Cylindric Partitions}
\author{Runqiao Li}
\address{[R.L.] Pennsylvania State University, Department of Mathematics, McAllister Building, 54 McAllister St, State College, PA 16801, United States}
\email{rml5856@psu.edu or runqiaoli@outlook.com}

\author{Ali K. Uncu}
\address{[A.K.U.] University of Bath, Faculty of Science, Department of Computer Science, Bath, BA2 7AY, UK}
\email{aku21@bath.ac.uk}

\keywords{Polynomial Identities, Cylindric Partitions, MacMahon Analysis, Infinite Hierarchies}
  
\subjclass[2020]{11B65; Secondary 33F10, 11C08, 11P81, 11P82, 11P84, 05A10, 05A15, 05A17, 05A30}

\date{\today}
\thanks{Research of the second author has been partially supported by the Austrian Science Fund FWF, P34501N, and UKRI EPSRC EP/T015713/1 projects.}

\begin{document}
\begin{abstract}
    We study cylindric partitions with two-element profiles using MacMahon's partition analysis. We find explicit formulas for the generating functions of the number of cylindric partitions by first finding the recurrences using partition analysis and then solving them. We also note some $q$-series identities related to these objects that show the manifestly positive nature of some alternating series. We generalize the proven identities and conjecture new polynomial refinements of Andrews-Gordon and Bressoud identities, which are companions to Foda--Quano's refinements. Finally, using a variant of the Bailey lemma, we present many new infinite hierarchies of polynomial identities.
\end{abstract}

\dedicatory{Commemorating the 85th birthdays of George E. Andrews and Bruce C. Berndt}
\maketitle

\section{Introduction}

In the spring of 1998, Andrews initiated a project at RISC on using the MacMahon computational method to solve problems in connection with linear homogeneous diophantine inequalities and equations \cite{OmegaWeb}. This collaboration has been incredibly fruitful for the partitions theory community. Through a four-decade-spanning collaboration, Andrews and Paule \cite{Andrews1, Andrews2, Andrews8, Andrews13, OmegaPackage} (some joint with Riese) studied many interesting partition classes using MacMahon's Partition Analysis. In these works, MacMahon's set up of homogeneous linear equations for partitions of interest is the starting point and Riese's implementation of the  Mathematica package Omega \cite{OmegaPackage, OmegaWeb}, which assists in carrying out MacMahon analysis' tedious and repetitive reduction formulas automatically, is the main tool\footnote{For recent advancements in MacMahon analysis, interested readers can check \cite{BreuerZaf} and refer to \texttt{SageMATH} and \texttt{Julia} implementations \cite{PolyhedralOmega, Zaf3}.}. 

We briefly define integer partitions and generating functions, and give some associated notations. A partitions $\pi$ is a finite list of non-increasing positive integers $(\lambda_1,\lambda_2,\dots, \lambda_{\#(\pi)}),$ where $\#(\pi)$ denotes the number of elements in the partition $\pi$. We call these elements \textit{parts} of the partition. The total of all the parts of $\pi$ is called the {\it size} of $\pi$ and is denoted by $|\pi|$. As a convention, we consider the empty list the unique partition with 0 parts and 0 size. For a given sequence of numbers $\{a_n\}_{n\geq 0}$ and a formal variable $q$, we say $\sum_{n\geq 0 } a_n q^n$ is the generating function of $a_n$ (or $q$-series associated with $a_n$). 

The partition analysis approach commonly follows four steps: 1) Pick a class of partitions to enumerate. 2) Write the defining conditions on these partitions as linear inequalities and equalities. 3) Fix the number of parts of partitions (or some other statistic) and write a (rough) rational generating function associated that enumerates the partitions when some linear inequalities are satisfied. 4) Applying the reduction rules (of MacMahon and many others \cite{MacMahon,Andrews8}) modifies (simplifies) the rational generating function to a (refined) rational generating function of only the partitions of interest. The operator that refines the rough generating function is called the Omega operator. The partition analysis for a partition of $n$ parts generally starts with $n$ variables, one for each part, and some number of parameters (to be eliminated) that encode the conditions on the partitions. After applying the Omega operator, the parameters get out of the picture. One takes the refined ratinal generating functions back to $q$-series by substituting $q$ for every variable. 

Cylindric partitions \cite{GesselKrattenthaler} (to be defined in Section~\ref{sec:background}) have been of recent interest \cite{KR, CDU, U, W, CW, Shunsuke, BU, Halime1, Halime2}. Most of these recent works focus on proving sum-product identities. This paper aims to study cylindric partitions with two-element profiles using partition analysis. Kurşungöz--Ömrüuzun Seyrek's papers \cite{Halime1,Halime2} on nicely decomposing the cylindric partitions are the closest related among the recent works. However, the techniques used and the outcomes are different. Those papers do not focus on fixing the number of elements. Whereas, in partition analysis, these restrictions are inherent and the results gained reflect that. For example,

\begin{theorem}\label{thm:RR_Gens}
    Let $n\geq0$ and $a\geq b \geq 0$ be integers and let $CP_{(a,b)}(n)$ be the generating function for the number of cylindric partitions with profile $(a,b)$ where each partition in the cylindric partition has at most $n$ parts. We always take $CP_{(a,b)}(0)=1$ as the cylindric partition of the empty lists.
    Then \begin{align}\label{eq:CP21n}
        CP_{(2,1)}(n) &= \frac{1}{(q;q)_{2n}} \sum_{r=-\infty}^\infty (-1)^r q^{r(5r+1)/2} {2n \brack n-\frac{5r}{2} + \frac{(-1)^r-1}{4}}_q,\\\label{eq:CP30n}
        CP_{(3,0)}(n) &= \frac{1}{(q;q)_{2n}} \sum_{r=-\infty}^\infty (-1)^r q^{r(5r+3)/2} {2n \brack n-\frac{5r}{2} +  3\frac{(-1)^r-1}{4}}_q,
    \end{align}
    where the $q$-Pochhammer symbol $(a;q)_n$ and the $q$-binomial coefficients ${a\brack b}_q = (q;q)_a/ ((q;q)_b(q;q)_{a-b})$ are defined as in \cite{A}.
\end{theorem}

Readers familiar with partition theory can recognize that, as $n\rightarrow \infty$, these generating functions are related to the Rogers--Ramanujan identities with an extra factor $1/(q;q)_\infty$ coming from the limit of the $q$-binomial coefficient. Similarly, recall the well-known finite version of the Rogers--Ramanujan identity \cite[Exercise 10, p. 50]{A}: For $a=0$ and $1$, \begin{equation}\label{eq:FinRR}\sum_{j\geq 0} q^{j^2 + a j} {n+1-a-j\brack j}_q = \sum_{r=-\infty}^\infty (-1)^r q^{r(5r+1)/2 - 2 a r} {n+1\brack \lfloor \frac{n-5r+1}{2}\rfloor +a}_q\end{equation} due to Andrews, where $\lfloor\cdot\rfloor$ is the floor function. It is easy to see that after the substitutions $(a,n)\mapsto(0,2n-1)$, the right-hand sides sums of \eqref{eq:CP21n} and \eqref{eq:FinRR} match. However, the sum on the right-hand side of \eqref{eq:CP30n} does not match Andrews' finite sum when $a=1$ under any substitutions.

Gordon's generalization of the Rogers--Ramanujan identities has the following analytic form due to Andrews \cite{A} \[\sum_{n_1\geq  \dots\geq n_{k-1}\geq 0} \frac{q^{n_1^2+n_2^2+\dots+n_{k-1}^2 + n_i+\dots +n_{k-1}}}{(q;q)_{n_1-n_2}(q;q)_{n_2-n_3}\dots(q;q)_{n_{k-2}-n_{k-1}}(q;q)_{n_{k-1}}} = \frac{1}{(q;q)_\infty} \sum_{r=-\infty}^{\infty} (-1)^r q^{\frac{r( (2k+1)r + 2k-2i +1)}{2}}.\] These identities are commonly presented in sum-product identities form, where the right-hand side is given as $(q^i,q^{2k+1-i},q^{2k+1};q^{2k+1})_\infty / (q;q)_\infty$. This is just one Jacobi Triple Product identity \cite[Theorem 2.8]{A} application away. Foda and Quano discovered a polynomial refinement of these identities \cite[Theorem 1.1]{FodaQuano}. We copy it here with $n \mapsto 2n$.

\begin{theorem}[Foda--Quano, 1994]\label{thm:FodaQuano} Let $n,k,i$ be fixed integers where $n\geq 0$ and $k\geq i\geq 1$ Then, \begin{align}
    \nonumber\sum_{n_1\geq \dots\geq n_{k-1}\geq n_k=0} q^{n_1^2+n_2^2+\dots+n_{k-1}^2 + n_i+\dots +n_{k-1}} \prod_{j=1}^{k-1} &{2n- 2\sum_{l=1}^{j-1}n_l -n_j - n_{j+1}  - \alpha_{ij}\brack n_j - n_{j+1}}_q \\ \label{eq:FodaQuano_FinAndrewsGordon} &= \sum_{r=-\infty}^{\infty} (-1)^r q^{\frac{r( (2k+1)r + 2k-2i +1)}{2}} {2n\brack   n -\lfloor \frac{i-k-(2k+1)r}{2}\rfloor }_q,
\end{align}
where $\alpha_{ij} := \max\{j-i+1,0\}$.
\end{theorem}

Andrews' finite Rogers-Ramanujan identity \eqref{eq:FinRR} appears in Foda--Quano infinite hierarchy of polynomial identities when $k=2$. After studying cylindric partitions with small profiles we discovered a companion hierarchy to \eqref{eq:FodaQuano_FinAndrewsGordon}. 

\begin{conjecture}\label{conj1}
    Let $n,k,i$ be integers where $n\geq 0$, $k\geq 5$, and $k > i\geq 1$, then \begin{align}
    \nonumber\sum_{n_1\geq \dots\geq n_{k-1}\geq n_k=0} q^{n_1^2+n_2^2+\dots+n_{k-1}^2 + n_i+\dots +n_{k-1}} &\prod_{j=1}^{k-1} {2n- 2\sum_{l=1}^{j-1}n_l -n_j - n_{j+1} -2 \alpha_{ij} \brack n_j - n_{j+1}}'_q \\ \label{eq:LiUncu_FinAndrewsGordon} &\hspace{-1.5cm}= \sum_{r=-\infty}^{\infty} (-1)^r q^{\frac{r( (2k+1)r + 2k-2i +1)}{2}} {2n\brack n - \frac{(2k+1)r}{2} + (2k-2i+1)\frac{(-1)^r-1}{4}}_q,
    \end{align}
    where $\alpha_{ij}:= \max\{j-i+1,0\}$ as in Theorem~\ref{thm:FodaQuano} and ${a\brack b}'_q$ is the ordinary $q$-binomial coefficient, except when $a<0$ and $b=0$, it is defined as 1.
\end{conjecture}

The infinite hierarchy \eqref{eq:LiUncu_FinAndrewsGordon} matches \eqref{eq:FodaQuano_FinAndrewsGordon} only when $k=i$. Hence, these cases are omitted in the Conjecture~\ref{conj1}. The left-hand sides of both identities seem almost identical, except for the doubled contribution of the matrix entries. Moreover, we have the following theorem.

\begin{theorem}\label{thm:LiUncu_AndrewsGordon_Thm} The equation \eqref{eq:LiUncu_FinAndrewsGordon} holds for every $n\geq0$, where, $ k=2,3$, and $4$, $ k\geq i\geq 1$. Furthermore, for $k=2,3$ and $i=1,2$, we have
\[CP_{(2k-i,i-1)}(n) := \frac{1}{(q;q)_{2n}} \sum_{r=-\infty}^{\infty} (-1)^r q^{\frac{r( (2k+1)r + 2k-2i +1)}{2}} {2n\brack n - \frac{(2k+1)r}{2} + (2k-2i+1)\frac{(-1)^r-1}{4}}_q.\]
\end{theorem}

Also note that when $k=i=1$, the left-hand sides of \eqref{eq:FodaQuano_FinAndrewsGordon} and \eqref{eq:LiUncu_FinAndrewsGordon} are empty sums. However, studying the right-hand sides yield the formula \begin{equation}\label{eq:Rogers}\sum_{r=-\infty}^\infty (-1)^r q^{r(3r+1)/2} {2n \brack n - \lfloor\frac{3r}{2} \rfloor}=1 \end{equation} which is due to Rogers \cite{Rogers}. 

Similar to Andrews--Gordon analytic identities, Bressoud identities are given as follows \cite{Bressoud} \[\sum_{n_1\geq  \dots\geq n_{k-1}\geq 0} \frac{q^{n_1^2+n_2^2+\dots+n_{k-1}^2 + n_i+\dots +n_{k-1}}}{(q;q)_{n_1-n_2}(q;q)_{n_2-n_3}\dots(q;q)_{n_{k-2}-n_{k-1}}(q^2;q^2)_{n_{k-1}}} = \frac{1}{(q;q)_\infty} \sum_{r=-\infty}^{\infty} (-1)^r q^{r( kr +k-i)}.\] Once again, polynomial refinements of these identities were noted before.

\begin{theorem}[Foda-Quano, 1994]\label{thm:FodaQuano2} Let $n,k,i$ be fixed integers where $n\geq 0$, $k\geq 2$, and $k\geq i\geq 1$ Then, \begin{align}
    \nonumber\sum_{n_1\geq \dots\geq n_{k-1}\geq0} q^{n_1^2+n_2^2+\dots+n_{k-1}^2 + n_i+\dots +n_{k-1}} {n-\sum_{j=1}^{k-2}n_j \brack n_{k-1}}_{q^2} \prod_{j=1}^{k-2} &{2n- 2\sum_{l=1}^{j-1}n_l -n_j - n_{j+1}  + \beta^{(k)}_{ij}\brack n_j - n_{j+1}}_q \\ \label{eq:FodaQuano_FinBressoud} &= \sum_{r=-\infty}^{\infty} (-1)^r q^{r( kr + k-i)} {2n+k-i\brack n -kr }_q,
\end{align}
where $\beta^{(k)}_{ij}:= \min\{k-i,k-j-1\}$.
\end{theorem}

\begin{conjecture}\label{conj2} Let $n, k, i$ be non-negative integers, where $k\geq 5$, $k > i\geq 1$, then
\begin{align}
    \nonumber\sum_{n_1\geq \dots\geq n_{k-1}\geq 0} q^{n_1^2+\dots+n_{k-1}^2 + n_i+\dots +n_{k-1}}& {n-\sum_{j=1}^{k-2}n_j - k +i\brack n_{k-1}}'_{q^2} \prod_{j=1}^{k-2} {2n- 2\sum_{l=1}^{j-1}n_l -n_j - n_{j+1}  - 2\alpha_{ij}\brack n_j - n_{j+1}}'_q \\ \label{eq:LiUncu_FinBressoud} &\hspace{1cm}= \sum_{r=-\infty}^{\infty} (-1)^r q^{r( kr + k-i)} {2n\brack n -kr +(k-i) \frac{(-1)^r-1}{2}}_q,
\end{align}
where $\alpha_{ij} := \max\{j-i+1,0\}$.
\end{conjecture}

\begin{theorem}\label{thm:LU_Bressoud_k23} For integers $n\geq 0$, $k=2,3$ and $4$, where $k\geq i\geq 1$,  \eqref{eq:LiUncu_FinBressoud} holds.
\end{theorem}

It is worth mentioning that recently Berkovich found many new companions to Bressoud identities \cite{BerkovichEvenModuli}.

Once again, when $k=i$, \eqref{eq:FodaQuano_FinBressoud} and \eqref{eq:LiUncu_FinBressoud} match. For the $k=i=1$ case right-hand side of \eqref{eq:LiUncu_FinBressoud} gives \begin{equation}\label{eq:LU_Bressoud_k1}\sum_{r=-\infty}^{\infty} (-1)^r q^{r^2} {2n\brack n -r }_q = (q;q^2)_n,\end{equation} which can be proven by \cite[(3.3.8), p. 37]{A} with $q\mapsto q^{-1}$ and $j=n-r$.

The abovementioned identities will also help us prove and conjecture many more infinite families of polynomial identities. Some of these equations refine Andrews' and Foda--Quano's earlier results. One of such identities is the following, which is a finite refinement of Andrews' \cite[(3.4)]{AndrewsMultiple}.

\begin{theorem}\label{thm:Finite_Rogers_IntroThm}
For $n\geq0$ and $p\geq 1$ integers, we have
    \begin{align}\nonumber\sum_{m_p\geq \dots\geq m_{1}\geq 0} &\frac{q^{m_p^2+m_{p-1}^2 +\dots+m_1^2}(q;q)_{2n}}{(q;q)_{n-m_p}(q;q)_{m_p - m_{p-1}}\dots (q;q)_{m_2-m1}(q;q)_{2m_1}}
     \\\label{eq:Finite_Rogers_IntroThm}&\hspace{3cm}= \sum_{r=-\infty}^{\infty} (-1)^r q^{\frac{r( 3r +1)}{2}+p\left(\frac{3r}{2} - \frac{(-1)^r-1}{4}\right)^2} {2n\brack n - \frac{3r}{2} + \frac{(-1)^r-1}{4}}_q.
     \end{align}
\end{theorem}

The structure of the rest of the paper is as follows. In Section~\ref{sec:background}, for completeness, we give the necessary definitions of the functions to be used. We also briefly introduce cylindric partitions and MacMahon analysis. Section~$n$, where $n=\ $\ref{sec:ab2}, \ref{sec:ab3}, \ref{sec:ab4}, and \ref{sec:ab5}, are reserved for studying cylindric partitions using MacMahon analysis into two element profiles with total $n-1$, respectively. The $q$-series theorems found and proven in these sections are the $k=2$ and $3$ cases of Theorems~\ref{thm:LiUncu_AndrewsGordon_Thm} and \ref{thm:LU_Bressoud_k23}. In Section~\ref{sec:conjs}, we state a couple more related conjectures and discuss possible ways to prove Conjectures~\ref{conj1} and \ref{conj2} that we plan to pursue. The last section, Section~\ref{sec:newHierarchies}, is reserved for Bailey lemma applications to the results presented in the Introduction. This yields many more infinite hierarchies; some proven and some conjectural.

\section{Background}\label{sec:background}

We will be using standard definitions of $q$-Pochhammer symbols \cite{A}
\[(a;q)_n = \prod_{j=0}^{n-1}(1-a q^j)\quad \text{and} \quad (a_1,a_2,\dots,a_k;q)_n := \prod_{j=1}^k (a_j;q)_n,\]
where $n \in \mathbb{Z}_{\geq0} \cup \{\infty\}$. An important side-note here is that we have \[ \frac{1}{(q;q)_n} = 0\] for every $n\leq 0$. The $q$-binomial coefficients are defined as \[{a\brack b}_q :=\frac{(q;q)_a}{(q;q)_b(q;q)_{a-b}},\] where $a$ and $b$ are integers, and they satisfy the recurrence \[{a\brack b}_q = {a-1\brack b}_q + q^{a-b}{a-1\brack b-1}_q.\] In combinatorics uses, $q$-binomial coefficients hardly ever receive negative arguments. Moreover, due to the partition theoretic interpretation of the $q$-binomial coefficients as partitions in a box with size $b \times (a-b)$, we tend to define $q$-binomials with negative arguments 0. However, the above mentioned recurrence applied to ${0\brack 0}_q =1$ requires a choice between ${-1\brack 0}_q$ and ${-1\brack -1}_q$ to be selected as 1. Furthermore, this issue propagates. To that end, we will define \[{a\brack b}'_q := \left\{ \begin{array}{ll}
1, & \text{if } a<0\text{ and }b=0,\\
{a\brack b}_q, & \text{otherwise.}
\end{array}\right\}. \]

The objects we will study in this paper are cylindric partitions. Gessel and Krattenthaler introduced them in 1997 \cite{GesselKrattenthaler}.  We give the formal definition here.

\begin{definition}\label{def:cylin} Let $k$ and $\ell$ be positive integers. Let $c=(c_1,c_2,\dots, c_k)$ be a composition, where $c_1+c_2+\dots+c_k=\ell$. A \emph{cylindric partition with profile $c$} is a vector partition $\Lambda = (\lambda^{(1)},\lambda^{(2)},\dots,\lambda^{(k)})$, where each $\lambda^{(i)} = \lambda^{(i)}_1+\lambda^{(i)}_2 + \cdots +\lambda^{(i)}_{s_i}$ is a partition, such that for all $i$ and $j$,
$$\lambda^{(i)}_j\geq \lambda^{(i+1)}_{j+c_{i+1}} \quad \text{and} \quad \lambda^{(k)}_{j}\geq\lambda^{(1)}_{j+c_1}.$$
\end{definition}

We extend the definition of the size of a partition to cylindric partitions without changing the notation. It is defined as the total of the sizes that make up the cylindric partition: $|\Lambda|$ of a cylindric partition $\Lambda = (\lambda^{(1)},\lambda^{(2)},\dots,\lambda^{(k)})$ is $|\lambda^{(1)}|+|\lambda^{(2)}|+\dots+|\lambda^{(k)}|$.

In 2007, Borodin \cite{Borodin} showed the generating functions for the number of cylindric partitions of a fixed profile has a product representation. 

\begin{theorem}[Borodin, 2007]
\label{th:Borodin}
Let $k$ and $\ell$ be positive integers, and let $c=(c_1,c_2,\dots,c_k)$ be a composition of $\ell$. Define $t:=k+\ell$ and $s(i,j) := c_i+c_{i+1}+\dots+ c_j$. Then,
\begin{equation*}
\label{BorodinProd}
F_c(q) = \frac{1}{(q^t;q^t)_\infty} \prod_{i=1}^k \prod_{j=i}^k \prod_{m=1}^{c_i} \frac{1}{(q^{m+j-i+s(i+1,j)};q^t)_\infty} \prod_{i=2}^k \prod_{j=2}^i \prod_{m=1}^{c_i} \frac{1}{(q^{t-m+j-i-s(j,i-1)};q^t)_\infty},
\end{equation*} where $F_c(q)$ is the generating function for the number of cylindric partitions where the exponent of $q$ keeps track of the size of these objects.
\end{theorem}

Corteel--Welsh's recurrence relations \cite{CW} for the refined generating functions for cylindric partitions, where a new parameter keeps track of the largest part among the cylindric partitions, led to many studies of these objects and some groundbreaking discoveries in the recent years. Here, continuing this trend, we study these combinatorial objects with small profiles using partition analysis.

The partition analysis is a method to study the generating functions introduced by MacMahon in his Combinatorial Analysis \cite{MacMahon}. It relies on the Omega operator which was defined as follows.
\begin{definition}
The Omega operator $\Omega_{\geq}$ is given by
$$\underset{\geq}{\Omega}\sum_{s_1=-\infty}^{\infty}\cdots\sum_{s_r=-\infty}^{\infty}A_{s_1,\ldots,s_r}\lambda_1^{s_1}\cdots\lambda_r^{s_r}:=\sum_{s_1=0}^{\infty}\cdots\sum_{s_r=0}^{\infty}A_{s_1,\ldots,s_r},$$
where the domain of the $A_{s_1\ldots,s_r}$ is the field of rational functions over $\mathbb{C}$
in several complex variables and the $\lambda_i$ are restricted to a neighborhood
of the circle $|\lambda_i|=1$. 

In addition, the $A_{s_1\ldots,s_r}$ are required to be such that
any of the series involved is absolute convergent within the domain of the
definition of $A_{s_1\ldots,s_r}$.
\end{definition}

Loosely speaking, the Omega operator would delete all the terms with at least one negative power of some $\lambda_i$, and then send all the $\lambda$'s to $1$. To apply the Omega operator, it is custom to start with the 'crude form' of the generating function and then eliminate the $\lambda$'s by some proper reduction rules. There are various such rules, and a number of them, which are frequently used, were listed in \cite{Andrews2}. For the subject of this paper, we will need the following.
\begin{lemma}\label{EliminationRule}
The following are some elimination rules for the Omega operator.
\begin{equation}\label{Elimination1}
\underset{\geq}{\Omega}\frac{1}{(1-A\lambda)}=\frac{1}{1-A},   
\end{equation}
\begin{equation}\label{Elimination2}
\underset{\geq}{\Omega}\frac{1}{(1-A\lambda)(1-\frac{B}{\lambda})}=\frac{1}{(1-A)(1-AB)},
\end{equation}
\begin{equation}\label{Elimination3}
\underset{\geq}{\Omega}\frac{1}{(1-A\lambda)(1-B\lambda)(1-\frac{C}{\lambda})}=\frac{1-ABC}{(1-A)(1-B)(1-AC)(1-BC)},  
\end{equation}
\begin{equation}\label{Elimination4}
\underset{\geq}{\Omega}\frac{1}{(1-A\lambda)(1-\frac{B\lambda}{\mu})(1-C\mu)(1-\frac{D\mu}{\lambda})}=\frac{1-ABCD}{(1-A)(1-C)(1-AD)(1-BC)(1-BD)}.   
\end{equation}
\end{lemma}

\begin{proof}
The equation (\ref{Elimination2}) and (\ref{Elimination3}) can be found in \cite{Andrews2}, while \eqref{Elimination1} is a special case of \eqref{Elimination2} when $B=0$. As for (\ref{Elimination4}), by iterating (\ref{Elimination3}) twice we have
\begin{align*}
\underset{\geq}{\Omega}\frac{1}{(1-A\lambda)(1-\frac{B\lambda}{\mu})(1-C\mu)(1-\frac{D\mu}{\lambda})}=&\frac{1-ABD}{(1-A)(1-BD)}\underset{\geq}{\Omega}\frac{1}{(1-AD\mu)(1-C\mu)(1-\frac{B}{\mu})}\\
=&\frac{1-ABD}{(1-A)(1-BD)}\times\frac{1-ABCD}{(1-AD)(1-C)(1-ABD)(1-BC)}\\
=&\frac{1-ABCD}{(1-A)(1-C)(1-AD)(1-BC)(1-BD)}.
\end{align*}
So we finish the proof.
\end{proof}
We will frequently applying these identities without mention for the main result of profile $(1,1)$ and $(2,0)$. And before moving on, we would like to list common notations we will use while doing MacMahon analysis.
\begin{itemize}
    \item $\lambda$'s and $\mu$'s are the parameters to be eliminated by the Omega operator.
    \item $X_i:=x_1x_2\cdots x_i$.
    \item $Y_i:=y_1y_2\cdots y_i$.
\end{itemize}

\section{Cylindric Partitions with profile $(c_1,c_2)$ such that $c_1+c_2=2$}\label{sec:ab2}

\subsection{Profile $(1,1)$}
We start by considering the Cylindric partitions with profile $(1,1)$. The following diagram indicates such a cylindric partition with at most $n$ nonzero entries in each row.

{\small\begin{center}
$\begin{ytableau}
\none & \none & b_{1} & b_{2} & b_{3} & \cdots & b_{n}\\
\none & a_{1} & a_{2} & a_{3} & \cdots & a_{n}\\
b_{1} & b_{2} & b_{3} & \cdots & b_{n}
\end{ytableau}$
\end{center}}

Let $\mathcal{CP}_{(1,1)}(n)$ be the set of all cylindric partitions with profile $(1,1)$ with at most $n$ entries in each row. Define
$$CP_{(1,1)}(n,X,Y):=\sum_{\pi\in\mathcal{CP}_{(1,1)}(n)}x_1^{a_1}x_2a^{2}\cdots x_n^{a_n}y_1^{b_1}y_2^{b_2}\cdots y_n^{b_n}$$
as the generating function for such partitions. Then by the partition analysis, we have
\begin{align*}
CP_{(1,1)}(n,X,Y)=\underset{\geq}{\Omega}\sum_{\substack{a_1,a_2,\ldots,a_n\\b_1,b_2,\ldots,b_n}}\prod_{i=1}^{n}x_i^{a_i}y_i^{b_i}\prod_{i=1}^{n}\lambda_{1,i}^{a_{i}-a_{i+1}}\lambda_{2,i}^{b_{i}-b_{i+1}}\prod_{i=1}^{n}\mu_{1,i}^{a_{i}-b_{i+1}}\mu_{2,i}^{b_{i}-a_{i+1}}.
\end{align*}
The range of the summation is for each of the $a_i$'s and $b_i$'s ($i\leq n$) to go through all the nonnegative integers, while $a_i$ and $b_i$ is taken to be $0$ for $i>n$.
Here both the $\lambda$'s and $\mu$'s are the extra variables to be eliminated by the Omega operator. Since we are dealing with $2$-dimensional partitions, it would be natural to use different letters to separate the inequalities along the rows and the columns. Thus here the $\lambda_{1,i}$'s and $\lambda_{2,i}$'s represent the weakly decreasing order in the first row and second row, while the $\mu_{1,i}$'s and $\mu_{2,i}$'s indicate the weakly decreasing order in the columns.

The next step is to eliminate all the extra variables in the "crude form" and get a closed form of the generating function. To do that, we will start by proving the following recurrence relation.

\begin{theorem}
For any $n>1$, the generating function for $\mathcal{CP}_{(1,1)}(n)$ satisfies
$$CP_{(1,1)}(n,X,Y)=\frac{(1-X_{n-1}X_{n}Y_{n-1}Y_{n})}{(1-X_{n}Y_{n-1})(1-X_{n-1}Y_{n})(1-X_{n}Y_{n})}\times CP_{(1,1)}(n-1,X,Y).$$    
\end{theorem}

\begin{proof}
By the crude form we have and Lemma \ref{EliminationRule},
\begin{align*}
CP_{(1,1)}(n,X,Y)=&
\underset{\geq}{\Omega}\sum_{\substack{a_1,a_2,\ldots,a_n\\b_1,b_2,\cdots,b_n}}\prod_{i=1}^{n}x_i^{a_i}y_i^{b_i}\prod_{i=1}^{n}\lambda_{1,i}^{a_{i}-a_{i+1}}\lambda_{2,i}^{b_{i}-b_{i+1}}\prod_{i=1}^{n}\mu_{1,i}^{a_{i}-b_{i+1}}\mu_{2,i}^{b_{i}-a_{i+1}}\\
=&
\underset{\geq}{\Omega}\frac{1}{(1-x_1\lambda_{1,1}\mu_{1,1})}\prod_{i=2}^{n-1}\left(1-\frac{x_i\lambda_{1,i}\mu_{1,i}}{\lambda_{1,i-1}\mu_{2,i-1}}\right)^{-1}\left(1-\frac{x_n}{\lambda_{1,n-1}\mu_{2,n-1}}\right)^{-1}\\
&\times\frac{1}{(1-y_1\lambda_{2,1}\mu_{2,1})}\prod_{i=2}^{n-1}\left(1-\frac{y_i\lambda_{2,i}\mu_{2,i}}{\lambda_{2,i-1}\mu_{1,i-1}}\right)^{-1}\left(1-\frac{y_n}{\lambda_{2,n-1}\mu_{1,n-1}}\right)^{-1}\\
=&
\underset{\geq}{\Omega}\frac{1}{(1-X_1\mu_{1,1})}\prod_{i=2}^{n-1}\left(1-\frac{X_i\mu_{1,1}\cdots\mu_{1,i}}{\mu_{2,1}\cdots\mu_{2,i-1}}\right)^{-1}\left(1-\frac{X_n\mu_{1,1}\cdots\mu_{1,n-1}}{\mu_{2,1}\cdots\mu_{2,n-1}}\right)^{-1}\\
&\times
\frac{1}{(1-Y_1\mu_{2,1})}\prod_{i=2}^{n-1}\left(1-\frac{Y_i\mu_{2,1}\cdots\mu_{2,i}}{\mu_{1,1}\cdots\mu_{1,i-1}}\right)^{-1}\left(1-\frac{Y_n\mu_{2,1}\cdots\mu_{2,n-1}}{\mu_{1,1}\cdots\mu_{1,n-1}}\right)^{-1}\\
=&
\underset{\geq}{\Omega}\frac{1}{(1-X_1\mu_{1,1})(1-Y_1\mu_{2,1})}\prod_{i=2}^{n-2}\left(1-\frac{X_i\mu_{1,1}\cdots\mu_{1,i}}{\mu_{2,1}\cdots\mu_{2,i-1}}\right)^{-1}\prod_{i=2}^{n-2}\left(1-\frac{Y_i\mu_{2,1}\cdots\mu_{2,i}}{\mu_{1,1}\cdots\mu_{1,i-1}}\right)^{-1}\\
&\times
\left(1-\frac{X_{n-1}\mu_{1,1}\cdots\mu_{1,n-1}}{\mu_{2,1},\cdots\mu_{2,n-2}}\right)^{-1}\left(1-\frac{Y_{n-1}\mu_{2,1}\cdots\mu_{2,n-1}}{\mu_{1,1}\cdots\mu_{1,n-2}}\right)^{-1}\\
&\times
\left(1-\frac{X_n\mu_{1,1}\cdots\mu_{1,n-1}}{\mu_{2,1}\cdots\mu_{2,n-1}}\right)^{-1}\left(1-\frac{Y_n\mu_{2,1}\cdots\mu_{2,n-1}}{\mu_{1,1}\cdots\mu_{1,n-1}}\right)^{-1}\\
=&
\underset{\geq}{\Omega}\frac{1}{(1-X_1\mu_{1,1})(1-Y_1\mu_{2,1})}\prod_{i=2}^{n-2}\left(1-\frac{X_i\mu_{1,1}\cdots\mu_{1,i}}{\mu_{2,1}\cdots\mu_{2,i-1}}\right)^{-1}\prod_{i=2}^{n-2}\left(1-\frac{Y_i\mu_{2,1}\cdots\mu_{2,i}}{\mu_{1,1}\cdots\mu_{1,i-1}}\right)^{-1}\\
&\times
\left(1-\frac{X_{n-1}\mu_{1,1}\cdots\mu_{1,n-2}}{\mu_{2,1},\cdots\mu_{2,n-2}}\right)^{-1}\left(1-\frac{Y_{n-1}\mu_{2,1}\cdots\mu_{2,n-2}}{\mu_{1,1}\cdots\mu_{1,n-2}}\right)^{-1}\\
&\times\frac{1-X_{n-1}X_{n}Y_{n-1}Y_{n}}{(1-X_{n-1}Y_n)(1-X_nY_{n-1})(1-X_nY_n)}\\
=&CP_{(1,1)}(n-1,X,Y)\times\frac{1-X_{n-1}X_{n}Y_{n-1}Y_{n}}{(1-X_{n-1}Y_n)(1-X_nY_{n-1})(1-X_nY_n)}.
\end{align*}
So we finish the proof.
\end{proof}
By evaluating the crude form for $n=1$, we get$$f_{1}=\frac{1}{(1-X_1)(1-Y_1)}.$$
So we have the following result.

\begin{theorem}
The generating function for $\mathcal{CP}_{(1,1)}(n)$ is given by
$$CP_{(1,1)}(n,X,Y)=\prod_{i=1}^{n}\frac{(1-X_{i-1}X_{i}Y_{i-1}Y_{i})}{(1-X_{i}Y_{i-1})(1-X_{i-1}Y_{i})(1-X_{i}Y_{i})}.$$
\end{theorem}

By Borodin's theorem, we should have
$$\sum_{\lambda\in\mathcal{CP}_{(1,1)}}q^{|\lambda|}=\frac{1}{(q,q,q^3,q^3,q^4;q^4)_{\infty}}=\frac{(-q;q^2)_{\infty}}{(q;q)_{\infty}}.$$
Now if we set $x_1=x_2=\cdots=y_1=y_2=\cdots=q$, then $X_i=Y_i=q^{i}$, we would have
$$f_{n}=\frac{(-q;q^2)_n}{(q;q)_{2n}}$$
in general. Letting $n\to\infty$, it matches Borodin's theorem as we expected.

\subsection{Profile $(2,0)$}
The following diagram indicates such a cylindric partition with at most $n$ nonzero entries in each row.
{\small\begin{center}
$\begin{ytableau}
\none & \none & b_{1} & b_{2} & b_{3} & \cdots & b_{n}\\
a_{1} & a_{2} & a_{3} & \cdots & a_{n}\\
b_{1} & b_{2} & b_{3} & \cdots & b_{n}
\end{ytableau}$
\end{center}}
and the generating function is given by
\begin{align*}CP_{(2,0)}(n,X,Y)=&\underset{\geq}{\Omega}\sum_{\substack{a_1,\ldots,a_n\\b_1,\ldots,b_n}}\prod_{i=1}^{n}x_{i}^{a_i}y_{i}^{b_i}\prod_{i=1}^{n}\lambda_{1,i}^{a_{i}-a_{i+1}}\lambda_{2,i}^{b_{i}-b_{i+1}}\prod_{i=1}^{n}\mu_{1,i}^{a_{i}-b_{i}}\mu_{2,i}^{b_{i}-a_{i+2}}\\
=&
\underset{\geq}{\Omega}\frac{1}{(1-x_1\lambda_{1,1}\mu_{1,1})}\left(1-\frac{x_2\lambda_{1,2}\mu_{1,2}}{\lambda_{1,1}}\right)^{-1}\prod_{i=3}^{n}\left(1-\frac{x_i\lambda_{1,i}\mu_{1,i}}{\lambda_{1,i-1}\mu_{2,i-2}}\right)^{-1}\\
&\times
\left(1-\frac{y_1\lambda_{2,1}\mu_{2,1}}{\mu_{1,1}}\right)^{-1}\prod_{i=2}^{n}\left(1-\frac{y_i\lambda_{2,i}\mu_{2,i}}{\lambda_{2,i-1}\mu_{1,i}}\right)^{-1}.
\end{align*}
And the initial value is given by
$$
g_1=\frac{1}{(1-x_1)(1-x_1y_1)}=\frac{1}{(1-X_1)(1-X_1Y_1)}.
$$

\begin{theorem}
For any $n>1$, the generating function for $\mathcal{CP}_{(2,0)}(n)$ satisfies
$$CP_{(2,0)}(n,X,Y)=\frac{1-X_{n-1}X_{n}Y_{n-2}Y_{n-1}}{(1-X_{n}Y_{n-2})(1-X_{n}Y_{n-1})(1-X_{n}Y_{n})}\times CP_{(2,0)}(n-1,X,Y).$$       
\end{theorem}

\begin{proof}
By the crude form above and Lemma \ref{EliminationRule},
\begin{align*}
&CP_{(2,0)}(n,X,Y)\\
=&
\underset{\geq}{\Omega}\frac{1}{(1-x_1\lambda_{1,1}\mu_{1,1})}\left(1-\frac{x_2\lambda_{1,2}\mu_{1,2}}{\lambda_{1,1}}\right)^{-1}\prod_{i=3}^{n}\left(1-\frac{x_i\lambda_{1,i}\mu_{1,i}}{\lambda_{1,i-1}\mu_{2,i-2}}\right)^{-1}\\
&\times
\left(1-\frac{y_1\lambda_{2,1}\mu_{2,1}}{\mu_{1,1}}\right)^{-1}\prod_{i=2}^{n}\left(1-\frac{y_i\lambda_{2,i}\mu_{2,i}}{\lambda_{2,i-1}\mu_{1,i}}\right)^{-1}\\
=&
\underset{\geq}{\Omega}\frac{1}{(1-X_1\mu_{1,1})(1-X_2\mu_{1,1}\mu_{1,2})}\prod_{i=3}^{n}\left(1-\frac{X_i\mu_{1,1}\cdots\mu_{1,i}}{\mu_{2,1}\cdots\mu_{2,i-2}}\right)^{-1}\\
&\times\prod_{i=1}^{n-2}\left(1-\frac{Y_i\mu_{2,1}\cdots\mu_{2,i}}{\mu_{1,1}\cdots\mu_{1,i}}\right)^{-1}\left(1-\frac{Y_{n-1}\mu_{2,1}\cdots\mu_{2,n-2}}{\mu_{1,1}\cdots\mu_{1,n-1}}\right)^{-1}\left(1-\frac{Y_n\mu_{2,1}\cdots\mu_{2,n-2}}{\mu_{1,1}\cdots\mu_{1,n}}\right)^{-1}\\
=&
\frac{1}{(1-X_nY_n)}\underset{\geq}{\Omega}\frac{1}{(1-X_1\mu_{1,1})(1-X_2\mu_{1,1}\mu_{1,2})}\prod_{i=3}^{n-2}\left(1-\frac{X_i\mu_{1,1}\cdots\mu_{1,i}}{\mu_{2,1}\cdots\mu_{2,i-2}}\right)^{-1}\\
&\times\prod_{i=1}^{n-3}\left(1-\frac{Y_i\mu_{2,1}\cdots\mu_{2,i}}{\mu_{1,1}\cdots\mu_{1,i}}\right)^{-1}\left(1-\frac{X_{n-1}\mu_{1,1}\cdots\mu_{1,n-1}}{\mu_{2,1}\cdots\mu_{2,n-3}}\right)^{-1}\left(1-\frac{X_n\mu_{1,1}\cdots\mu_{1,n-1}}{\mu_{2,1}\cdots\mu_{2,n-2}}\right)^{-1}\\
&\times\left(1-\frac{Y_{n-2}\mu_{2,1}\cdots\mu_{2,n-2}}{\mu_{1,1}\cdots\mu_{1,n-2}}\right)^{-1}\left(1-\frac{Y_{n-1}\mu_{2,1}\cdots\mu_{2,n-2}}{\mu_{1,1}\cdots\mu_{1,n-1}}\right)^{-1}\\
=&\frac{1}{1-X_nY_n}\underset{\geq}{\Omega}\frac{1}{(1-X_1\mu_{1,1})(1-X_2\mu_{1,1}\mu_{1,2})}\prod_{i=3}^{n-2}\left(1-\frac{X_i\mu_{1,1}\cdots\mu_{1,i}}{\mu_{2,1}\cdots\mu_{2,i-2}}\right)^{-1}\\
&\times\prod_{i=1}^{n-3}\left(1-\frac{Y_i\mu_{2,1}\cdots\mu_{2,i}}{\mu_{1,1}\cdots\mu_{1,i}}\right)^{-1}\left(1-\frac{X_{n-1}\mu_{1,1}\cdots\mu_{1,n-2}}{\mu_{2,1}\cdots\mu_{2,n-3}}\right)^{-1}\left(1-\frac{Y_{n-2}\mu_{2,1}\cdots\mu_{2,n-3}}{\mu_{1,1}\cdots\mu_{1,n-2}}\right)^{-1}\\
&\times\frac{1-X_{n-1}X_nY_{n-2}Y_{n-1}}{(1-X_nY_{n-2})(1-X_nY_{n-1})(1-X_{n-1}Y_{n-1})}\\
=&CP_{(2,0)}(n-1,X,Y)\times\frac{1-X_{n-1}X_nY_{n-2}Y_{n-1}}{(1-X_nY_{n-2})(1-X_nY_{n-1})(1-X_{n-1}Y_{n-1})}.
\end{align*}
So we finish the proof.
\end{proof}
Now with the value for $CP_{(2,0)}(1)$, we have the following.
\begin{theorem}
The generating function for $\mathcal{CP}_{(2,0)}(n)$ is given by
$$CP_{(2,0)}(n,X,Y)=\prod_{i=1}^{n}\frac{1-X_{n-1}X_{n}Y_{n-2}Y_{n-1}}{(1-X_{n}Y_{n-2})(1-X_{n}Y_{n-1})(1-X_{n}Y_{n})}.$$
\end{theorem}

By Borodin's theorem, the generating function for cylindric partitions with profile $(2,0)$ is
$$\sum_{\lambda\in\mathcal{CP}_{(2,0)}}q^{|\lambda|}=\frac{1}{(q,q^2,q^2,q^3,q^4;q^4)_{\infty}}=\frac{(-q^2;q^2)_{\infty}}{(q;q)_{\infty}}.$$
In our calculation, if we set $x_1=x_2=\cdots=y_1=y_2=\cdots=q$, then $X_{n}=Y_{n}=q^{n}$ for $n\geq1$ and $X_{n}=Y_{n}=1$ otherwise, we would have
$$CP_{(2,0)}(n)=\frac{(-q^2;q^2)_{n-1}}{(q;q)_{2n}}.$$
Letting $n\to\infty$, we have the Borodin's theorem for profile $(2,0)$.

\subsection{$k=2$ cases of \eqref{eq:LiUncu_FinBressoud}}

As a part of the proof of Theorem~\ref{thm:LU_Bressoud_k23}, we write down the \eqref{eq:LiUncu_FinBressoud} identities when $k=2$:

\begin{align}
 \label{eq:k2i1} \sum_{n_1\geq 0} q^{n_1^2+n_1} {n-1\brack n_1}'_{q^2} &= \sum_{r=-\infty}^{\infty} (-1)^r q^{2r^2+r}{2n\brack n-2r +\frac{(-1)^r-1}{2}}_q,\\
\label{eq:k2i2}  \sum_{n_1\geq 0} q^{n_1^2} {n\brack n_1}'_{q^2} &= \sum_{r=-\infty}^{\infty} (-1)^r q^{2r^2}{2n\brack n-2r}_q.
\end{align}

We prove these identities by showing that both sides of (3.i) satisfy the same recurrences \begin{equation}
    \label{eq:recCP20} a_i(n+1) = (1 + q^{2 n + i-1})\ a_i(n)
\end{equation} where $i=1$ and $2$. Also, these sequences have the same initial values $a_i(0)=1$. 

Moreover, for $i=1$ and 2, it is easy to check that that $a_i/(q;q)_{2n}$ satisfies the same recurrences and the same initial conditions as $CP_{(3-i,i-1)}(n)$ and $CP_{(3-i,i-1)}(n)$, respectively. Therefore, we can explicitly write

\begin{align}
\label{eq:CP20n_GF}
CP_{(2,0)}(n) &= \frac{1}{(q;q)_{2n}}\sum_{r=-\infty}^{\infty} (-1)^r q^{2r^2+r}{2n\brack n-2r +\frac{(-1)^r-1}{2}}_q,\\
\label{eq:CP11n_GF}
CP_{(1,1)}(n) &= \frac{1}{(q;q)_{2n}}\sum_{r=-\infty}^{\infty} (-1)^r q^{2r^2}{2n\brack n-2r}_q.
\end{align}

\section{Cylindric Partitions with profile $(c_1,c_2)$ such that $c_1+c_2=3$}\label{sec:ab3}

\subsection{Profile $(2,1)$}
We first consider the cylindric partitions with profile $(2,1)$. This corresponds to the first Rogers-Ramanujan identity.
$$\frac{1}{(q;q)_{\infty}}\sum_{n=0}^{\infty}\frac{q^{n^2}}{(q;q)_{n}}=\frac{1}{(q;q)_{\infty}(q;q^5)_{\infty}(q^4;q^5)_{\infty}}.$$
The following diagram indicates a cylindric partition with profile $(2,1)$ and at most $n$ entries in each row.
{\small\begin{center}
$\begin{ytableau}
\none & \none & \none & b_{1} & b_{2} & b_{3} & \cdots & b_{n}\\
\none & a_{1} & a_{2} & a_{3} & \cdots & a_{n}\\
b_{1} & b_{2} & b_{3} & \cdots & b_{n}
\end{ytableau}$
\end{center}}
And we would have the following 'crude form'.
\begin{align*}
CP_{(2,1)}(n,X,Y)=&
\underset{\geq}{\Omega}\sum_{\substack{a_1,\ldots,a_n\\b_1,\ldots,b_n}}\prod_{i=1}^{n}x_{i}^{a_i}y_{i}^{b_i}\prod_{i=1}^{n}\lambda_{1,i}^{a_i-a_{i+1}}\lambda_{2,i}^{b_i-b_{i+1}}\prod_{i=1}^{n}\mu_{1,i}^{a_i-b_{i+1}}\mu_{2,i}^{b_i-a_{i+2}}\\
=&\frac{1}{(1-x_1\lambda_{1,1}\mu_{1,1})}\times\frac{1}{(1-\frac{x_2\lambda_{1,2}\mu_{1,2}}{\lambda_{1,1}})}\prod_{i=3}^{n}\frac{1}{(1-\frac{x_i\lambda_{1,i}\mu_{1,i}}{\lambda_{1,i-1}\mu_{2,i-2}})}\\
&\times
\frac{1}{(1-\frac{y_1\lambda_{2,1}\mu_{2,1}}{\mu_{1,1}})}\prod_{i=2}^{n}\frac{1}{(1-\frac{y_i\lambda_{2,i}\mu_{2,i}}{\lambda_{2,i-1}\mu_{1,i-1}})}.
\end{align*}
We get the following initial values by exactly calculating this expression and reducing it to the $q$-case for some small value of $n$.
\begin{align*}
CP_{(2,1)}(1)=&\frac{1}{(q;q)_2}(1+q),\\
CP_{(2,1)}(2)=&\frac{1}{(q;q)_4}\left(1+q+q^2+q^3+q^4\right).
\end{align*}
Next, we present a recurrence relation satisfied by $CP_{(2,1)}(n)$.
\begin{theorem}\label{CP21recurrence}
The generating function for $\mathcal{CP}_{(2,1)}(n)$ satisfies 
\begin{align*}
CP_{(2,1)}(n)=&\frac{1+q^{2n-1}+q^{2n-2}}{(1-q^{2n-1})(1-q^{2n})}CP_{(2,1)}(n-1)\\
&-\frac{q^{4n-5}}{(1-q^{2n-3})(1-q^{2n-2})(1-q^{2n-1})(1-q^{2n})}CP_{(2,1)}(n-2).
\end{align*}
\end{theorem}

\begin{proof}
Starting with a cylindric partition in $\mathcal{CP}_{(2,1)}(n)$, we consider the following three cases.
\begin{align*}
&\begin{pmatrix}
 & & & b_1 & \cdots & b_{n-4} & b_{n-3} & b_{n-2} \\
& a_1 & a_2 & a_3 & \cdots & a_{n-2} & a_{n-1} & a_{n} \\
b_1 & b_{2} & b_3 & b_4 & \cdots & b_{n-1} & b_{n} & 
\end{pmatrix}
\end{align*}

\begin{itemize}
    \item Case $1$. $b_{n-2}\geq b_{n-1}\geq b_{n}\geq a_{n}$.

    We first subtract $a_{n}$ from each entry and then subtract $b_n-a_n$ from the remaining entries.
    \begin{align*}
&\begin{pmatrix}
 \cdots & b_{n-4} & b_{n-3} & b_{n-2} \\
\cdots & a_{n-2} & a_{n-1} & a_{n} \\
 \cdots & b_{n-1} & b_{n} & 
\end{pmatrix}\\
\longrightarrow&
\begin{pmatrix}
 \cdots & b_{n-4}-a_{n} & b_{n-3}-a_{n} & b_{n-2}-a_{n} \\
\cdots & a_{n-2}-a_{n} & a_{n-1}-a_{n} & 0 \\
 \cdots & b_{n-1}-a_{n} & b_{n}-a_{n} & 
\end{pmatrix}\\
\longrightarrow&
\begin{pmatrix}
 \cdots & b_{n-4}-b_{n} & b_{n-3}-b_{n} & b_{n-2}-b_{n} \\
\cdots & a_{n-2}-b_{n} & a_{n-1}-b_{n} & 0 \\
 \cdots & b_{n-1}-b_{n} & 0 & 
\end{pmatrix}
\end{align*}
We end up with a cylindric partition in $\mathcal{CP}_{(2,1)}(n-1)$ and this gives the factor $\frac{1}{(1-q^{2n-1})(1-q^{2n})}$.
\item Case $2$. $b_{n-2}\geq b_{n-1}\geq a_{n}>b_{n}$.

We subtract $1$ from each entry except $b_n$ then interchange $a_{n}-1$ and $b_{n}$. Now we go back to Case $1$. and do the same operation.
    \begin{align*}
&\begin{pmatrix}
 \cdots & b_{n-4} & b_{n-3} & b_{n-2} \\
\cdots & a_{n-2} & a_{n-1} & a_{n} \\
 \cdots & b_{n-1} & b_{n} & 
\end{pmatrix}\\
\longrightarrow&
\begin{pmatrix}
 \cdots & b_{n-4}-1 & b_{n-3}-1 & b_{n-2}-1 \\
\cdots & a_{n-2}-1 & a_{n-1}-1 & a_{n}-1 \\
 \cdots & b_{n-1}-1 & b_{n} & 
\end{pmatrix}\\
\longrightarrow&
\begin{pmatrix}
 \cdots & b_{n-4}-1 & b_{n-3}-1 & b_{n-2}-1 \\
\cdots & a_{n-2}-1 & a_{n-1}-1 & b_{n} \\
 \cdots & b_{n-1}-1 & a_{n}-1 & 
\end{pmatrix}\\
\longrightarrow&
\begin{pmatrix}
 \cdots & b_{n-4}-1-b_{n} & b_{n-3}-1-b_{n} & b_{n-2}-1-b_{n} \\
\cdots & a_{n-2}-1-b_{n} & a_{n-1}-1-b_{n} & 0 \\
 \cdots & b_{n-1}-1-b_{n} & a_{n}-1-b_{n} & 
\end{pmatrix}\\
\longrightarrow&
\begin{pmatrix}
 \cdots & b_{n-4}-a_{n} & b_{n-3}-a_{n} & b_{n-2}-a_{n} \\
\cdots & a_{n-2}-a_{n} & a_{n-1}-a_{n} & 0 \\
 \cdots & b_{n-1}-a_{n} & 0 & 
\end{pmatrix}
\end{align*}
Now we get a cylindric partition in $\mathcal{CP}_{(2,1)}(n-1)$ and this gives the factor $\frac{q^{2n-1}}{(1-q^{2n-1})(1-q^{2n})}$.
\item Case $3$. $b_{n-2}\geq a_{n}>b_{n-1}\geq b_{n}$.

We subtract $1$ from each entry except $b_{n-1}$ and $b_n$ then permute $a_{n}-1$, $b_{n-1}$ and $b_{n}$. Now we go back to Case $1$. and do the same operation.
    \begin{align*}
&\begin{pmatrix}
 \cdots & b_{n-4} & b_{n-3} & b_{n-2} \\
\cdots & a_{n-2} & a_{n-1} & a_{n} \\
 \cdots & b_{n-1} & b_{n} & 
\end{pmatrix}\\
\longrightarrow&
\begin{pmatrix}
 \cdots & b_{n-4}-1 & b_{n-3}-1 & b_{n-2}-1 \\
\cdots & a_{n-2}-1 & a_{n-1}-1 & a_{n}-1 \\
 \cdots & b_{n-1} & b_{n} & 
\end{pmatrix}\\
\longrightarrow&
\begin{pmatrix}
 \cdots & b_{n-4}-1 & b_{n-3}-1 & b_{n-2}-1 \\
\cdots & a_{n-2}-1 & a_{n-1}-1 & b_{n} \\
 \cdots & a_{n}-1 & b_{n-1} & 
\end{pmatrix}\\
\longrightarrow&
\begin{pmatrix}
 \cdots & b_{n-4}-1-b_{n} & b_{n-3}-1-b_{n} & b_{n-2}-1-b_{n} \\
\cdots & a_{n-2}-1-b_{n} & a_{n-1}-1-b_{n} & 0 \\
 \cdots & a_{n}-1-b_{n} & b_{n-1}-b_{n} & 
\end{pmatrix}\\
\longrightarrow&
\begin{pmatrix}
 \cdots & b_{n-4}-1-b_{n-1} & b_{n-3}-1-b_{n-1} & b_{n-2}-1-b_{n-1} \\
\cdots & a_{n-2}-1-b_{n-1} & a_{n-1}-1-b_{n-1} & 0 \\
 \cdots & a_{n}-1-b_{n-1} & 0 & 
\end{pmatrix}
\end{align*}
We get a cylindric partition in $\mathcal{CP}_{(2,1)}(n-1)$ and this gives the factor $\frac{q^{2n-2}}{(1-q^{2n-1})(1-q^{2n})}$.
\end{itemize}
Note that in the last case, since $a_n\leq a_{n-1}$, so the resulted cylindric partition is just a proper subset of $\mathcal{CP}_{(2,1)}(n-1)$. Now for a cylindric partition in $\mathcal{CP}_{(2,1)}(n-1)$ with $b_{n-1}>a_{n-1}$, we apply the following operation.
\begin{align*}
&\begin{pmatrix}
 \cdots & b_{n-4} & b_{n-3} & b_{n-2} \\
\cdots & a_{n-2} & a_{n-1} & 0 \\
 \cdots & b_{n-1} & 0 & 
\end{pmatrix}\\
\longrightarrow&
\begin{pmatrix}
\cdots & b_{n-4}-1 & b_{n-3}-1 & b_{n-2}-1 \\
\cdots & a_{n-2}-1 & a_{n-1} & 0 \\
 \cdots & b_{n-1}-1 & 0 &
\end{pmatrix}\\
\longrightarrow&
\begin{pmatrix}
\cdots & b_{n-4}-1-a_{n-1} & b_{n-3}-1-a_{n-1} & b_{n-2}-1-a_{n-1} \\
\cdots & a_{n-2}-1-a_{n-1} & 0 & 0 \\
 \cdots & b_{n-1}-1-a_{n-1} & 0 &
\end{pmatrix}\\
\longrightarrow&
\begin{pmatrix}
\cdots & b_{n-4}-b_{n-1} & b_{n-3}-b_{n-1} & b_{n-2}-b_{n-1} \\
\cdots & a_{n-2}-b_{n-1} & 0 & 0 \\
 \cdots & 0 & 0 &
\end{pmatrix}
\end{align*}
We end up with an arbitrary cylindric partition in $\mathcal{CP}_{(2,1)}(n-2)$ and the weight been deleted is generated by $\frac{q^{2n-3}}{(1-q^{2n-3})(1-q^{2n-2})}$.

Putting them together, we have
\begin{align*}
CP_{(2,1)}(n)=&
\frac{1+q^{2n-1}}{(1-q^{2n-1})(1-q^{2n})}CP_{(2,1)}(n-1)\\
&+\frac{q^{2n-2}}{(1-q^{2n-1})(1-q^{2n})}\left(CP_{(2,1)}(n-1)-\frac{q^{2n-3}}{(1-q^{2n-3})(1-q^{2n-2})}CP_{(2,1)}(n-2)\right)\\
=&\frac{1+q^{2n-1}+q^{2n-2}}{(1-q^{2n-1})(1-q^{2n})}CP_{(2,1)}(n-1)\\
&-\frac{q^{4n-5}}{(1-q^{2n-3})(1-q^{2n-2})(1-q^{2n-1})(1-q^{2n})}CP_{(2,1)}(n-2).
\end{align*}
So we finish the proof.   
\end{proof}
Here and in the sequel, for any profile $(c_1,c_2)$, define
$$CP_{(c_1,c_2)}(n)=:\frac{P_{(c_1,c_2)}(n)}{(q;q)_{2n}}.$$
Then the recurrence of $CP_{(2,1)}(n)$ implies the following recurrence for $P_{(2,1)}(n)$.
\begin{theorem}
The $P_{(2,1)}(n)$ satisfies
\begin{equation}\label{eq:P21rec} P_{(2,1)}(n)  = (1+q^{2n-1} + q^{2n-2}) P_{(2,1)}(n-1) - q^{4n-5} P_{(2,1)}(n-2),  \end{equation}
with the initial conditions $P_{(2,1)}(0)=1$ and $P_{(2,1)}(1)=1+q$.
\end{theorem}

By evaluating the initial terms of $P_{(2,1)}(n)$ and using the fitting functions in \cite{qFuncs} it is easy to see that a closed formula for $P_{(2,1)}(n)$ is the following:
\begin{theorem}\label{thm:P21formula}
\begin{equation}\label{eq:P21formula}P_{(2,1)}(n) =\sum_{j=-\infty}^\infty q^{10j^2 +j}{2n \brack n-5j}_q  -\sum_{j=-\infty}^\infty q^{10j^2 +11j+3}{2n \brack n-5j-3}_q , \end{equation}
and consequently, we have
\begin{equation}\label{eq:CP21formula}
CP_{(2,1)}(n)=\frac{1}{(q;q)_{2n}}\left(\sum_{j=-\infty}^\infty q^{10j^2 +j}{2n \brack n-5j}_q  -\sum_{j=-\infty}^\infty q^{10j^2 +11j+3}{2n \brack n-5j-3}_q\right).
\end{equation}
\end{theorem}

\begin{proof}
Showing that the equation \eqref{eq:P21formula} satisfies \eqref{eq:P21rec} is a standard symbolic computation exercise. Using \texttt{qMultiSum} \cite{qMultiSum} we can find (and simultaneously prove) that the two individual sums in \eqref{eq:P21formula} satisfy the same 6th-order recurrence. Then, using \texttt{qFunctions} \cite{qFuncs}, we show that the greatest common divisor of this 6th recurrence and \eqref{eq:P21rec} is \eqref{eq:P21rec}. They also satisfy the same initial conditions, proving that equation \eqref{eq:P21formula} satisfies \eqref{eq:P21rec}. 
\end{proof}

Instead of first guessing the expression \eqref{eq:P21formula} and then proving it. It is also possible to reverse engineer (which requires guessing the right factorial basis), find and prove the expression simultaneously using the q-Factorial Basis method of Jimenez-Pastor and the second author \cite{qFact}. 

Note that one can combine the two sums of \eqref{eq:P21formula} into an alternating bilateral sum. This is the bilateral sum given in \eqref{eq:CP21n}: \[P_{(2,1)}(n)=\sum_{r=-\infty}^\infty (-1)^r q^{r(5r+1)/2} {2n \brack n-\frac{5r}{2} +  \frac{(-1)^r-1}{4}}_q,\] which is now proven. In the limit, using Jacobi Triple Product Identity \cite[Theorem 2.8]{A}, we get

\begin{align*}
   \nonumber\lim_{n\to \infty} P_{(2,1)}(n)
   =&\frac{1}{(q;q)_\infty} \sum_{j=-\infty}^{\infty}(-1)^{j}q^{\frac{5j^{2}+j}{2}}
   =\frac{(q^2;q^5)_{\infty}(q^3;q^5)_{\infty}(q^5;q^5)_{\infty}}{(q;q)_{\infty}}
   =\frac{1}{(q;q^5)_{\infty}(q^4;q^5)_{\infty}}.
\end{align*}
Hence,
$$\lim_{n\to\infty}CP_{(2,1)}(n)=\lim_{n\to\infty}\frac{P_{(2,1)}(n)}{(q;q)_{2n}}=\frac{1}{(q;q)_{\infty}(q;q^5)_{\infty}(q^4;q^5)_{\infty}},$$
which, as expected, matches Borodin's product formula \cite{Borodin} for cylindric partitions with profile $(2,1)$.

\subsection{Profile $(3,0)$}
A cylindric partition with profile $(3,0)$ and at most $n$ entries in each row can be represented as follows.
{\small\begin{center}
$\begin{ytableau}
\none & \none & \none & b_{1} & b_{2} & b_{3} & b_{4} & \cdots & b_{n}\\
a_{1} & a_{2} & a_{3} & a_{4} & \cdots & a_{n}\\
b_{1} & b_{2} & b_{3} & b_{4} & \cdots & b_{n}
\end{ytableau}$
\end{center}}

And the crude form generating function is given by
\begin{align*}
CP_{(3,0)}(n,X,Y)=&
\underset{\geq}{\Omega}\sum_{\substack{a_1,\ldots,a_n\\b_1,\ldots,b_n}}\prod_{i=1}^{n}x_{i}^{a_i}y_{i}^{b_i}\prod_{i=1}^{n}\lambda_{1,i}^{a_i-a_{i+1}}\lambda_{2,i}^{b_i-b_{i+1}}\prod_{i=1}^{n}\mu_{1,i}^{a_i-b_{i}}\mu_{2,i}^{b_i-a_{i+3}}\\
=&\frac{1}{(1-x_1\lambda_{1,1}\mu_{1,1})}\prod_{i=2}^{3}\frac{1}{(1-\frac{x_i\lambda_{1,i}\mu_{1,i}}{\lambda_{1,i-1}})}\prod_{i=4}^{n}\frac{1}{(1-\frac{x_i\lambda_{1,i}\mu_{1,i}}{\lambda_{1,i-1}\mu_{2,i-3}})}\\
&\times
\frac{1}{(1-\frac{y_1\lambda_{2,1}\mu_{2,1}}{\mu_{1,1}})}\prod_{i=2}^{n}\frac{1}{(1-\frac{y_i\lambda_{2,i}\mu_{2,i}}{\lambda_{2,i-1}\mu_{1,i}})}.
\end{align*}
By the Omega operator, we get the following initial conditions.
    \begin{align*}
    CP_{(3,0)}(1)=&\frac{1}{(q;q)_2},\\
    CP_{(3,0)}(2)=&\frac{1}{(q;q)_4}(1+q^2),\\
    CP_{(3,0)}(3)=&\frac{1}{(q;q)_6}(1+q^2+q^3+q^4+q^6).
    \end{align*}

\begin{theorem}\label{CP30recurrence}
The generating function for $\mathcal{CP}_{(3,0)}(n)$ satisfies
\begin{align*}
CP_{(3,0)}(n)=&\frac{1+q^{2n-3}+q^{2n-2}}{(1-q^{2n-1})(1-q^{2n})}CP_{(3,0)}(n-1)\\
&-\frac{q^{4n-7}}{(1-q^{2n-3})(1-q^{2n-2})(1-q^{2n-1})(1-q^{2n})}CP_{(3,0)}(n-2)    
\end{align*}
\end{theorem}

\begin{proof}

We start with an arbitrary $\pi\in\mathcal{CP}_{(3,0)}(n)$.
\begin{itemize}
    \item Step $1$. Subtract $b_{n}$ from each entry.
\begin{align*}
&\begin{pmatrix}
 & \cdots & b_{n-5} & b_{n-4} & b_{n-3} \\
a_1 & \cdots & a_{n-2} & a_{n-1} & a_{n} \\
b_1 & \cdots & b_{n-2} & b_{n-1} & b_{n} 
\end{pmatrix}\\
\longrightarrow&
\begin{pmatrix}
 & \cdots & b_{n-5}-b_{n} & b_{n-4}-b_{n} & b_{n-3}-b_{n} \\
a_1-b_{n} & \cdots & a_{n-2}-b_{n} & a_{n-1}-b_{n} & a_{n}-b_{n} \\
b_1-b_{n} & \cdots & b_{n-2}-b_{n} & b_{n-1}-b_{n} & 0 
\end{pmatrix}   
\end{align*}
This gives the factor $\frac{1}{1-q^{2n}}$ and from now on we may assume $b_{n}$ always be $0$.

\item Step $2$. Now we have to consider three different cases.
\begin{itemize}
    \item Case $1$. $b_{n-3}\geq b_{n-2}\geq b_{n-1}\geq a_{n}\geq0$. Then we subtract $a_{n}$ from each entry and end up with a partition in $CP_{(3,0)}(n-1)$.
\begin{align*}
&\begin{pmatrix}
 & \cdots & b_{n-5} & b_{n-4} & b_{n-3} \\
a_1 & \cdots & a_{n-2} & a_{n-1} & a_{n} \\
b_1 & \cdots & b_{n-2} & b_{n-1} & 0 
\end{pmatrix}\\
\longrightarrow&
\begin{pmatrix}
 & \cdots & b_{n-5}-a_{n} & b_{n-4}-a_{n} & b_{n-3}-a_{n} \\
a_1-a_{n} & \cdots & a_{n-2}-a_{n} & a_{n-1}-a_{n} & 0 \\
b_1-a_{n} & \cdots & b_{n-2}-a_{n} & b_{n-1}-a_{n} & 0 
\end{pmatrix}
\end{align*}
This gives the factor $\frac{1}{1-q^{2n-1}}$.

\item Case $2$. $b_{n-3}\geq b_{n-2}\geq a_{n}>b_{n-1}\geq0$. We first subtract $1$ from all the entries except $b_{n-1}$, then switch $a_{n}-1$ and $b_{n-1}$. Now we go back to Case $1$. again and subtract $b_{n-1}$ from each entry. We end up with a partition in $\mathcal{CP}_{(3,0)}(n-1)$.
\begin{align*}
&\begin{pmatrix}
 & \cdots & b_{n-5} & b_{n-4} & b_{n-3} \\
a_1 & \cdots & a_{n-2} & a_{n-1} & a_{n} \\
b_1 & \cdots & b_{n-2} & b_{n-1} & 0 
\end{pmatrix}\\
\longrightarrow&
\begin{pmatrix}
 & \cdots & b_{n-5}-1 & b_{n-4}-1 & b_{n-3}-1 \\
a_1-1 & \cdots & a_{n-2}-1 & a_{n-1}-1 & a_{n}-1 \\
b_1-1 & \cdots & b_{n-2}-1 & b_{n-1} & 0 
\end{pmatrix}\\
\longrightarrow&
\begin{pmatrix}
 & \cdots & b_{n-5}-1 & b_{n-4}-1 & b_{n-3}-1 \\
a_1-1 & \cdots & a_{n-2}-1 & a_{n-1}-1 & b_{n-1} \\
b_1-1 & \cdots & b_{n-2}-1 & a_{n}-1 & 0 
\end{pmatrix}\\
\longrightarrow&
\begin{pmatrix}
 & \cdots & b_{n-5}-1-b_{n-1} & b_{n-4}-1-b_{n-1} & b_{n-3}-1-b_{n-1} \\
a_1-1-b_{n-1} & \cdots & a_{n-2}-1-b_{n-1} & a_{n-1}-1-b_{n-1} & 0 \\
b_1-1-b_{n-1} & \cdots & b_{n-2}-1-b_{n-1} & a_{n}-1-b_{n-1} & 0 
\end{pmatrix}
\end{align*}
This gives the factor $\frac{q^{2n-2}}{1-q^{2n-1}}$

\item Case $3$. $b_{n-3}\geq a_{n}>b_{n-2}\geq b_{n-1}\geq0$. We first subtract $1$ from all the entries except $b_{n-2}$ and $b_{n-1}$, then permute $a_{n}-1$, $b_{n-2}$, $b_{n-1}$ as follow. Finally, subtract $b_{n-1}$ from each entry, and we end up with a partition in $\mathcal{CP}_{(3,0)}(n-1)$.
\begin{align*}
&\begin{pmatrix}
 & \cdots & b_{n-5} & b_{n-4} & b_{n-3} \\
a_1 & \cdots & a_{n-2} & a_{n-1} & a_{n} \\
b_1 & \cdots & b_{n-2} & b_{n-1} & 0 
\end{pmatrix}\\
\longrightarrow&
\begin{pmatrix}
 & \cdots & b_{n-5}-1 & b_{n-4}-1 & b_{n-3}-1 \\
a_1-1 & \cdots & a_{n-2}-1 & a_{n-1}-1 & a_{n}-1 \\
b_1-1 & \cdots & b_{n-2} & b_{n-1} & 0 
\end{pmatrix}\\
\longrightarrow&
\begin{pmatrix}
 & \cdots & b_{n-5}-1 & b_{n-4}-1 & b_{n-3}-1 \\
a_1-1 & \cdots & a_{n-2}-1 & a_{n-1}-1 & b_{n-1} \\
b_1-1 & \cdots & a_{n}-1 & b_{n-2} & 0 
\end{pmatrix}\\
\longrightarrow&
\begin{pmatrix}
 & \cdots & b_{n-5}-1-b_{n-1} & b_{n-4}-1-b_{n-1} & b_{n-3}-1-b_{n-1} \\
a_1-1-b_{n-1} & \cdots & a_{n-2}-1-b_{n-1} & a_{n-1}-1-b_{n-1} & 0 \\
b_1-1-b_{n-1} & \cdots & a_{n}-1-b_{n-1} & b_{n-2}-b_{n-1} & 0 
\end{pmatrix}
\end{align*}
This gives the factor $\frac{q^{2n-3}}{1-q^{2n-1}}$.
\end{itemize}
\end{itemize}
But note that in Case $3$., the cylindric partition we get is not arbitrary, we have $a_{n}\leq a_{n-1}$, thus $a_{n}-1-b_{n-1}\leq a_{n-1}-1-b_{n-1}$. So, there is an overcounting in
$$\frac{1+q^{2n-3}+q^{2n-2}}{(1-q^{2n-1})(1-q^{2n})}CP_{(3,0)}(n-1).$$
Now, giving a cylindric partition $\pi\in\mathcal{CP}_{(3,0)}(n-1)$ with $b_{n-2}>a_{n-1}$, we do the following operation.
\begin{align*}
&\begin{pmatrix}
& \cdots & b_{n-5} & b_{n-4} \\
a_1 & \cdots & a_{n-2} & a_{n-1} \\
b_1 & \cdots & b_{n-2} & b_{n-1} 
\end{pmatrix}\\
\longrightarrow&
\begin{pmatrix}
& \cdots & b_{n-5}-b_{n-1} & b_{n-4}-b_{n-1} \\
a_1-b_{n-1} & \cdots & a_{n-2}-b_{n-1} & a_{n-1}-b_{n-1} \\
b_1-b_{n-1} & \cdots & b_{n-2}-b_{n-1} & 0 
\end{pmatrix}\\
\longrightarrow&
\begin{pmatrix}
& \cdots & b_{n-5}-b_{n-1}-1 & b_{n-4}-b_{n-1}-1 \\
a_1-b_{n-1}-1 & \cdots & a_{n-2}-b_{n-1}-1 & a_{n-1}-b_{n-1} \\
b_1-b_{n-1}-1 & \cdots & b_{n-2}-b_{n-1}-1 & 0 
\end{pmatrix}\\
\longrightarrow&
\begin{pmatrix}
& \cdots & b_{n-5}-a_{n-1}-1 & b_{n-4}-a_{n-1}-1 \\
a_1-a_{n-1}-1 & \cdots & a_{n-2}-a_{n-1}-1 & 0 \\
b_1-a_{n-1}-1 & \cdots & b_{n-2}-a_{n-1}-1 & 0 
\end{pmatrix}
\end{align*}
We end up with an arbitrary cylindric partition in $\mathcal{CP}_{(3,0)}(n-2)$ and the total weight deleted from $\pi$ is generated by $\frac{q^{2n-4}}{(1-q^{2n-3})(1-q^{2n-2})}$.

So we have
\begin{align*}
CP_{(3,0)}(n)=&\frac{1+q^{2n-2}}{(1-q^{2n-1})(1-q^{2n})}CP_{(3,0)}(n-1)\\
&+\frac{q^{2n-3}}{(1-q^{2n-1})(1-q^{2n})}\left(CP_{(3,0)}(n-1)-\frac{q^{2n-4}}{(1-q^{2n-3})(1-q^{2n-2})}CP_{(3,0)}(n-2)\right)\\
=&\frac{1+q^{2n-3}+q^{2n-2}}{(1-q^{2n-1})(1-q^{2n})}CP_{(3,0)}(n-1)\\
&-\frac{q^{4n-7}}{(1-q^{2n-3})(1-q^{2n-2})(1-q^{2n-1})(1-q^{2n})}CP_{(3,0)}(n-2)
\end{align*}
as we desired.    
\end{proof}

Similarly, recall that \[CP_{(3,0)}(n) = \frac{P_{(3,0)}(n)}{(q;q)_{2n}},\] 
Theorem \ref{CP30recurrence} implies the following recurrence for $P_{(3,0)}(n)$.
\begin{theorem}
For $n\geq2$, we have
\begin{equation}\label{eq:P30rec} P_{(3,0)}(n)  = (1+q^{2n-3} + q^{2n-2}) P_{(3,0)}(n-1) - q^{4n-7} P_{(3,0)}(n-2),
\end{equation}
with the initial conditions $P_{(3,0)}(1)=1$ and $P_{(3,0)}(2)=1+q^2$.
\end{theorem}

\begin{theorem}
For any positive integer $n$, we have
\begin{equation}\label{eq:P30formula}P_{(3,0)}(n) =\sum_{j=-\infty }^\infty q^{10j^2 +3j}{2n \brack n-5j}_q  -\sum_{j=-\infty }^\infty  q^{10j^2 +13j+4}{2n \brack n-5j-4}_q,
\end{equation}
and consequently, we have
\begin{equation}\label{eq:CP30formula}
CP_{(3,0)}(n)=\frac{1}{(q;q)_{2n}}\left(\sum_{j=-\infty }^\infty q^{10j^2 +3j}{2n \brack n-5j}_q  -\sum_{j=-\infty }^\infty  q^{10j^2 +13j+4}{2n \brack n-5j-4}_q\right).
\end{equation}
\end{theorem}

The proof of this theorem follows the same steps of Theorem~\ref{thm:P21formula}. Moreover, once again, we can combine these two bilateral sums into a single alternating bilateral sum. This presentation of the sum is given in \eqref{eq:CP30n}: \[P_{(3,0)}(n)=\sum_{r=-\infty}^\infty (-1)^r q^{r(5r+3)/2} {2n \brack n-\frac{5r}{2} +  3\frac{(-1)^r-1}{4}}_q.\]

In the limit, using Jacobi Triple Product Identity \cite[Theorem 2.8]{A}, we see that \eqref{eq:P30formula} yields
\begin{align*}
   \nonumber\lim_{n\to \infty} P_{(3,0)}(n)
   =&\frac{1}{(q;q)_\infty} 
   \sum_{j=-\infty}^\infty (-1)^j q^{\frac{5j^2+3j}{2}}
   = \frac{(q;q^5)_{\infty}(q^4;q^5)_{\infty}(q^5;q^5)_{\infty}}{(q;q)_{\infty}}
   =\frac{1}{(q^2;q^5)_{\infty}(q^3;q^5)_{\infty}}.
\end{align*}
Therefore,
\[\lim_{n\rightarrow \infty} CP_{(3,0)}(n) = \lim_{n\rightarrow \infty} \frac{P_{(3,0)}(n)}{(q;q)_{2n}} = \frac{1}{(q;q)_{\infty}(q^2;q^5)_{\infty}(q^3;q^5)_{\infty}},\]
which matches Borodin's theorem \cite{Borodin} for cylindric partitions with profile $(3,0)$.

\subsection{A polynomial identity due to George Andrews}

In \cite{Andrews1}, Andrews proved the following polynomial identity.
\begin{theorem}[Andrews]\label{thm:Andrews}
If $\alpha=0$ or $\alpha=-1$, then \begin{align*}
\sum_{j=0}^{\infty}q^{j^2-\alpha j}{n+1+\alpha-j\brack j}_{q}=&\sum_{j=-\infty}^{\infty}(-1)^{j}q^{\frac{j(5j+1)}{2}+2\alpha j}{n+1 \brack \lfloor\frac{n+1-5j}{2}\rfloor-\alpha}_{q}    
\end{align*}  
\end{theorem}

If $n$ is odd (replace $n$ by $2n-1$) and $\alpha=0$, we have
\begin{align}
\label{eq:P21Andrews}\sum_{j=0}^{\infty}q^{j^2}{2n-j \brack j}_{q}=&\sum_{j=-\infty}^{\infty}(-1)^{j}q^{\frac{j(5j+1)}{2}}{2n \brack \lfloor n-\frac{5j}{2}\rfloor}_{q}
\intertext{by splitting the right-hand sum with $j\mapsto 2j$ and $j\mapsto 2j+1$ we directly get}
\nonumber=&\sum_{j=-\infty}^{\infty}q^{10j^2+j}{2n \brack n-5j}_{q}-\sum_{j=-\infty}^{\infty}q^{10j^2+11j+3}{2n \brack n-5j-3}_{q}
\end{align}
which is equal to $P_{(2,1)}(n)$ (see \eqref{eq:P21formula}). 

The left-hand side of \eqref{eq:P21Andrews} is {\it manifestly positive} (i.e. these polynomials have only non-negative coefficients). We do not know this a priori for $P_{(2,1)}(n)$. We know that $CP_{(2,1)}(n)$ has positive coefficients through that being the generating function of cylindric partitions with profile $(2,1)$, where there are at most $n$ non-zero parts in each row. However, the formula of $P_{(2,1)}(n)$, \eqref{eq:P21formula}, has differences of polynomials and we do not have a clear combinatorial description of $P_{(2,1)}(n)$ either. Therefore, we can only tell that $P_{(2,1)}(n)$ has positive coefficients through the identity \eqref{eq:P21Andrews}.

Similarly, we can find a manifestly positive formula for $P_{(3,0)}(n)$.
\begin{equation}\label{eq:P30Andrews}
    \sum_{j=0}^{\infty}q^{j^2+j}{2n-j-2 \brack j}'_{q}=\sum_{r=-\infty}^\infty (-1)^r q^{r(5r+3)/2} {2n \brack n-\frac{5r}{2} +  3\frac{(-1)^r-1}{4}}_q.
\end{equation}
This is proven by comparing recurrences and initial conditions using computer algebra.
 
Observe that \eqref{eq:P30formula} is not a direct outcome of Theorem~\ref{thm:Andrews}. One would first need to take $\alpha \mapsto -1$, and then $n\mapsto 2n-2$ to make the left-hand sides of Theorem~\ref{thm:Andrews} and \eqref{eq:P30Andrews}, but the right side of Theorem~\ref{thm:Andrews} would have a $q$-binomial coefficient which has $2n-1$ as its top argument.

\section{Cylindric Partitions with profile $(c_1,c_2)$ such that $c_1+c_2=4$}\label{sec:ab4}

\subsection{Profile $(4,0)$}
The following diagram indicates such a cylindric partition with profile $(4,0)$ and at most $n$ nonzero entries in each row.
{\small\begin{center}
$\begin{ytableau}
\none & \none & \none &\none & b_{1} & b_{2} & b_{3} & b_{4} & b_{5} & \cdots & b_{n}\\
a_{1} & a_{2} & a_{3} & a_{4} & a_{5} & \cdots & a_{n}\\
b_{1} & b_{2} & b_{3} & b_{4} & b_{5} & \cdots & b_{n}
\end{ytableau}$
\end{center}}
And we have the following crude form for the generating function for arbitrary $n$. 
\begin{align*}
CP_{(4,0)}(n,X,Y)=&\underset{\geq}{\Omega}\sum_{\substack{a_1,\ldots,a_n\\b_1,\ldots,b_n}}\prod_{i=1}^{n}x_i^{a_i}y_i^{b_i}\prod_{i=1}^{n}\lambda_{1,i}^{a_i-a_{i+1}}\lambda_{2,i}^{b_1-b_{i+1}}\prod_{i=1}^{n}\mu_{1,i}^{a_i-b_i}\mu_{2,i}^{b_i-a_{i+4}}\\
=&\underset{\geq}{\Omega}\frac{1}{1-x_1\lambda_{1,1}\mu_{1,1}}\prod_{i=2}^{4}\frac{1}{1-\frac{x_i\lambda_{1,i}\mu_{1,i}}{\lambda_{1,i-1}}}\prod_{i=5}^{n}\frac{1}{1-\frac{x_i\lambda_{1,i}\mu_{1,i}}{\lambda_{1,i-1}\mu_{2,i-4}}}\\
&\times\frac{1}{1-\frac{y_1\lambda_{2,1}\mu_{2,1}}{\mu_{1,1}}}\prod_{i=2}^{n}\frac{1}{1-\frac{y_i\lambda_{2,i}\mu_{2,i}}{\lambda_{2,i-1}\mu_{1,i}}}.
\end{align*}
The following initial values, which can be computed by the Omega operator, will be needed.
\begin{align*}
CP_{(4,0)}(1)=&\frac{1}{(q;q)_2},\\
CP_{(4,0)}(2)=&\frac{1}{(q;q)_4}(1+q^2),\\
CP_{(4,0)}(3)=&\frac{1}{(q;q)_6}(1+q^2+q^3+q^4+q^6).
\end{align*}
    
Now we give the recurrence relation for $CP_{(4,0)}(n)$.
\begin{theorem}\label{CP40recurrence}
For any $n\geq 3$, we have
\begin{align*}
CP_{(4,0)}(n)=&\frac{1+q^{2n-4}+q^{2n-3}+q^{2n-2}}{(1-q^{2n-1})(1-q^{2n})}CP_{(4,0)}(n-1)\\
&-\frac{q^{4n-7}+q^{4n-8}+q^{4n-9}}{(1-q^{2n-3})(1-q^{2n-2})(1-q^{2n-1})(1-q^{2n})}CP_{(4,0)}(n-2)\\
&-\frac{q^{4n-10}}{(1-q^{2n-4})(1-q^{2n-3})(1-q^{2n-2})(1-q^{2n-1})(1-q^{2n})}CP_{(4,0)}(n-3).
\end{align*}
\end{theorem}

The proof of the relation for $CP_{(4,0)}(n)$ is similar to what we have done for profile $(3,0)$ and $(2,1)$, except there are more cases to be considered. So we omit it. This implies that for $n\geq3$, we also have
\begin{align*}
P_{(4,0)}(n)=&(1+q^{2n-4}+q^{2n-3}+q^{2n-2})P_{(4,0)}(n-1)-(q^{4n-7}+q^{4n-8}+q^{4n-9})P_{(4,0)}(n-2)\\
&-(q^{4n-10}-q^{6n-15})P_{(4,0)}(n-3).
\end{align*}

Here, we can once again extract a formula for the coefficients and this leads to the formula
\begin{theorem}
For any $n\geq0$ we have
\begin{equation}\label{eq:P40formula} P_{(4,0)}(n)  = \sum_{r= -\infty}^\infty (-1)^r q^{r(3r+2)} {2 n\brack n-3r+(-1)^r-1}_q.\end{equation}    
\end{theorem}

The proof follows the same steps Theorem~\ref{thm:P21formula}. Moreover, this implies that \[ \lim_{n\rightarrow\infty} P_{(4,0)}(n)  = \frac{1}{(q^2,q^3,q^4;q^6)_\infty}\] using the Jacobi Triple product identity. 

This product appears among the modulo 6 Bressoud identities. As noted in the introduction, Foda and Quano found a polynomial refinement of the Bressoud identities (see Theorem~\ref{thm:FodaQuano2}).
The polynomial refinement of the Bressoud identity that corresponds to the product $1/(q^2,q^3,q^3;q^6)_\infty$ is 
\begin{equation}\label{eq:FQ_Bressoud_mod6}
    \sum_{n_1\geq n_2\geq 0} q^{n_1^2+n_2^2+n_1+n_2} {n-n_1\brack n_2}_{q^2} {2n-n_1-n_2+1\brack n_1-n_2}_q = \sum_{j=-\infty}^\infty (-1)^r q^{r(3r+2)} {2 n+2\brack n-3r}_q
\end{equation}

The right-hand sides of \eqref{eq:P40formula} and \eqref{eq:FQ_Bressoud_mod6} are similar except for the $q$-binomial coefficient. This raises the natural question of whether there is a left-hand side associated to the right-hand side of \eqref{eq:P40formula} which will show the positivity of these terms. Indeed, there is one:

\begin{theorem}\label{thm:ManPos_P40} For $n\geq 0$, we have
 \[   \sum_{n_1\geq n_2\geq 0} q^{n_1^2 +n_2^2+n_1+n_2} {n-n_1-2\brack n_2}'_{q^2}{2n -n_1-n_2-2\brack n_1-n_2}'_q  = \sum_{j=-\infty}^\infty (-1)^j q^{j(3j+2)} {2 n\brack n-3j-1+(-1)^j}_q,
 \]
 where ${a \brack b}'_q = 1$ if $a<0$ and $b=0$, and equal to the ordinary $q$-binomial coefficient ${a\brack b}_q$, otherwise.
\end{theorem}

This is a new refinement of Bressoud's identity for the $k=3$ and $i=1$ case. 

\subsection{Profile $(3,1)$}

The following diagram indicates such a cylindric partition with profile $(3,1)$ and at most $n$ nonzero entries in each row.
{\small\begin{center}
$\begin{ytableau}
\none & \none & \none &\none & b_{1} & b_{2} & b_{3} & b_{4} & b_{5} & \cdots & b_{n}\\
\none & a_{1} & a_{2} & a_{3} & a_{4} & a_{5} & \cdots & a_{n}\\
b_{1} & b_{2} & b_{3} & b_{4} & b_{5} & \cdots & b_{n}
\end{ytableau}$
\end{center}}
And we have the following crude form for the generating function for arbitrary $n$. 
\begin{align*}
CP_{(4,0)}(n,X,Y)=&\underset{\geq}{\Omega}\sum_{\substack{a_1,\ldots,a_n\\b_1,\ldots,b_n}}\prod_{i=1}^{n}x_i^{a_i}y_i^{b_i}\prod_{i=1}^{n}\lambda_{1,i}^{a_i-a_{i+1}}\lambda_{2,i}^{b_1-b_{i+1}}\prod_{i=1}^{n}\mu_{1,i}^{a_i-b_{i+1}}\mu_{2,i}^{b_i-a_{i+3}}\\
=&\underset{\geq}{\Omega}\frac{1}{1-x_1\lambda_{1,1}\mu_{1,1}}\prod_{i=2}^{3}\frac{1}{1-\frac{x_i\lambda_{1,i}\mu_{1,i}}{\lambda_{1,i-1}}}\prod_{i=4}^{n}\frac{1}{1-\frac{x_i\lambda_{1,i}\mu_{1,i}}{\lambda_{1,i-1}\mu_{2,i-3}}}\\
&\times\frac{1}{1-y_1\lambda_{2,1}\mu_{2,1}}\prod_{i=2}^{n}\frac{1}{1-\frac{y_i\lambda_{2,i}\mu_{2,i}}{\lambda_{2,i-1}\mu_{1,i-1}}}.
\end{align*}

By computing this for some small $n$, we have the following initial values.
\begin{align*}
CP_{(3,1)}(1)=&\frac{1}{(q;q)_2}(1+q)\\
CP_{(3,1)}(2)=&\frac{1}{(q;q)_4}(q^4+q^3+q^2+q+1)\\
CP_{(3,1)}(3)=&\frac{1}{(q;q)_6}(q^9+q^8+q^7+2 q^6+2 q^5+2 q^4+2 q^3+q^2+q+1)
\end{align*}

Using the same combinatorial idea as in previous sections, we have the following recurrence.
\begin{theorem}
For $n\geq3$, we have
\begin{align*}
CP_{(3,1)}(n)=&\frac{1+q^{2n-3}+q^{2n-2}+q^{2n-1}}{(1-q^{2n-1})(1-q^{2n})}CP_{(3,1)}(n-1)\\
&-\frac{q^{4n-5}+q^{4n-6}+q^{4n-7}}{(1-q^{2n-3})(1-q^{2n-2})(1-q^{2n-1})(1-q^{2n})}CP_{(3,1)}(n-2)\\
&-\frac{q^{4n-8}}{(1-q^{2n-5})(1-q^{2n-3})(1-q^{2n-2})(1-q^{2n-1})(1-q^{2n})}CP_{(3,1)}(n-3).
\end{align*}    
\end{theorem}
And similarly, this leads to the following recurrence for $P_{3,1}(n):= CP_(3,1)(n)/(q;q)_{2n}$.
\begin{align*}
P_{(3,1)}(n)=&(1+q^{2n-3}+q^{2n-2}+q^{2n-1})P_{(3,1)}(n-1)-(q^{4n-5}+q^{4n-6}+q^{4n-7})P_{(3,1)}(n-2)\\
&-(q^{4n-8}-q^{6n-12})P_{(3,1)}(n-3).
\end{align*}
\begin{theorem}
For $n\geq0$, we have
    \begin{equation}\label{eq:P31formula} P_{(3,1)}(n)  = \sum_{r= -\infty}^\infty (-1)^r q^{r(3r+1)} {2 n\brack n-3r+\frac{(-1)^r-1}{2}}_q.\end{equation}    
\end{theorem}
Once again the proof is automated and omitted here. Note that we have \[\lim_{n\rightarrow\infty}P_{(3,0)}(n)=\frac{1}{(q,q^3,q^5;q^6)_{\infty}},\] which is the modulo 6 Bressoud identities when $k=3$ and $i=2$.

Similar to Theorem~\ref{thm:ManPos_P40}, we can once again find and prove a manifestly positive series representation for $P_{(3,1)}(n)$.

\begin{theorem}\label{thm:ManPos_P31}
    \[   \sum_{n_1\geq n_2\geq 0} q^{n_1^2 +n_2^2+n_2} {n-n_1-1\brack n_2}'_{q^2}{2n -n_1-n_2\brack n_1-n_2}'_q  = \sum_{j=-\infty}^\infty (-1)^j q^{j(3j+1)} {2 n\brack n-3j-1+(-1)^j}_q.
 \]
\end{theorem}

\subsection{Profile $(2,2)$}
The following diagram indicates such a cylindric partition with profile $(2,2)$ and at most $n$ nonzero entries in each row.
{\small\begin{center}
$
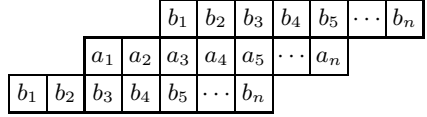
\begin{ytableau}
\none & \none & \none &\none & b_{1} & b_{2} & b_{3} & b_{4} & b_{5} & \cdots & b_{n}\\
\none & \none & a_{1} & a_{2} & a_{3} & a_{4} & a_{5} & \cdots & a_{n}\\
b_{1} & b_{2} & b_{3} & b_{4} & b_{5} & \cdots & b_{n}
\end{ytableau}$
\end{center}}
And the crude form for the generating function is as follows. 
\begin{align*}
CP_{(2,2)}(n,X,Y)=&\underset{\geq}{\Omega}\sum_{\substack{a_1,\ldots,a_n\\b_1,\ldots,b_n}}\prod_{i=1}^{n}x_i^{a_i}y_i^{b_i}\prod_{i=1}^{n}\lambda_{1,i}^{a_i-a_{i+1}}\lambda_{2,i}^{b_1-b_{i+1}}\prod_{i=1}^{n}\mu_{1,i}^{a_i-b_{i+2}}\mu_{2,i}^{b_i-a_{i+2}}\\
=&\underset{\geq}{\Omega}\frac{1}{1-x_1\lambda_{1,1}\mu_{1,1}}\times\frac{1}{1-\frac{x_2\lambda_{1,2}\mu_{1,2}}{\lambda_{1,1}}}\prod_{i=3}^{n}\frac{1}{1-\frac{x_i\lambda_{1,i}\mu_{1,i}}{\lambda_{1,i-1}\mu_{2,i-2}}}\\
&\times\frac{1}{1-y_1\lambda_{2,1}\mu_{2,1}}\times\frac{1}{1-\frac{\lambda_{2,2}\mu_{2,2}}{\lambda_{2,1}}}\prod_{i=3}^{n}\frac{1}{1-\frac{y_i\lambda_{2,i}\mu_{2,i}}{\lambda_{2,i-1}\mu_{1,i-2}}}.
\end{align*}
And we have the initial values are given by
\begin{align*}
CP_{(2,2)}(1)=&\frac{1}{(q;q)_{2}}(1+q),\\
CP_{(2,2)}(2)=&\frac{1}{(q;q)_{4}}(1+q+2q^{2}+q^{3}+q^{4}).  
\end{align*}

The limit case is $\frac{(q^3,q^3,q^6;q^6)_{\infty}}{(q;q)_{\infty}^{2}}$.

\begin{theorem}\label{thm:C22_rec}
The following recurrence relation of $CP_{(2,2)}(n)$ holds for $n\geq2$.
\begin{align*}
CP_{(2,2)}(n)=&\frac{1+q^{2n-2}+q^{2n-1}}{(1-q^{2n-1})(1-q^{2n})}CP_{(2,2)}(n-1)\\
&+\frac{q^{2n-2}(1-q^{2n-3})}{(1-q^{2n-3})(1-q^{2n-2})(1-q^{2n-1})(1-q^{2n})}CP_{(2,2)}(n-2),
\end{align*}
\end{theorem}
Again, this can be established by a similar combinatorial argument as in previous cases. Theorem~\ref{thm:C22_rec} implies the recurrence relation for $P_{(2,2)}(n)$. For $n\geq3$, we have
\begin{equation*}
P_{(2,2)}(n)=(1+q^{2n-2}+q^{2n-1})P_{(2,2)}(n-1)+q^{2n-2}(1-q^{2n-3})P_{(2,2)}(n-2).
\end{equation*}

\begin{theorem}
For $n\geq0$, we have
\begin{equation}\label{eq:P22formula} P_{(2,2)}(n)= \sum_{r= -\infty}^\infty (-1)^r q^{3r^2} {2n\brack n-3r}_{q},\end{equation} 
and it follows that
\begin{equation}\label{eq:CP22}
CP_{(2,2)}(n)=\frac{1}{(q;q)_{2n}}\sum_{r= -\infty}^\infty (-1)^r q^{3r^2} {2n\brack n-3r}_{q}.  
\end{equation}
\end{theorem}
Once again the proof is automated and omitted here. Note that we have \[\lim_{n\rightarrow\infty}P_{(3,0)}(n)=\frac{(q^3,q^3,q^6;q^6)_\infty}{(q;q)_{\infty}},\] which is the modulo 6 Bressoud identities when $k=i=3$.

Observe that the right-hand side of \eqref{eq:P22formula} matches the Foda--Quano's Bressoud refinement \eqref{eq:FodaQuano_FinBressoud} for $k=i=3$. Hence, we already know one manifestly positive representation of $P_{(2,2)}(n)$. Nevertheless, we can find another one in the same spirit with $P_{(4,0)}(n)$ and $P_{(3,1)}(n)$ (see Theorems~\ref{thm:ManPos_P40} and \ref{thm:ManPos_P31}). We present this as the following theorem.

\begin{theorem}\label{thm:ManPos_P22} 
    \[   \sum_{n_1\geq n_2\geq 0} q^{n_1^2 +n_2^2} {n-n_1\brack n_2}'_{q^2}{2n -n_1-n_2\brack n_1-n_2}'_q  = \sum_{j=-\infty}^\infty (-1)^j q^{3r^2} {2 n\brack n-3r}_q.
 \]
\end{theorem}

The identities \eqref{eq:CP11n_GF}, \eqref{eq:CP20n_GF}, Theorems~\ref{thm:ManPos_P40}, \ref{thm:ManPos_P31} and \ref{thm:ManPos_P22} together proves Theorem~\ref{thm:LU_Bressoud_k23}.

\section{Cylindric Partitions with profile $(c_1,c_2)$ such that $c_1+c_2=5$}\label{sec:ab5}
\subsection{Profile $(5,0)$}
The following diagram indicates such a cylindric partition with profile $(5,0)$ and at most $n$ nonzero entries in each row.
{\small\begin{center}
$\begin{ytableau}
\none & \none & \none & \none &\none & b_{1} & b_{2} & b_{3} & b_{4} & b_{5} & b_{6} & \cdots & b_{n}\\
a_{1} & a_{2} & a_{3} & a_{4} & a_{5} & a_{6} & \cdots & a_{n}\\
b_{1} & b_{2} & b_{3} & b_{4} & b_{5} & b_{6} & \cdots & b_{n}
\end{ytableau}$
\end{center}}
And we have the following crude form for the generating function.
\begin{align*}
CP_{(5,0)}(n,X,Y)=&\underset{\geq}{\Omega}\sum_{\substack{a_1,\ldots,a_n\\b_1,\ldots,b_n}}\prod_{i=1}^{n}x_i^{a_i}y_i^{b_i}\prod_{i=1}^{n}\lambda_{1,i}^{a_i-a_{i+1}}\lambda_{2,i}^{b_1-b_{i+1}}\prod_{i=1}^{n}\mu_{1,i}^{a_i-b_{i}}\mu_{2,i}^{b_i-a_{i+5}}\\
=&\underset{\geq}{\Omega}\frac{1}{1-x_1\lambda_{1,1}\mu_{1,1}}\prod_{i=2}^{5}\frac{1}{1-\frac{x_i\lambda_{1,i}\mu_{1,i}}{\lambda_{1,i-1}}}\prod_{i=6}^{n}\frac{1}{1-\frac{x_i\lambda_{1,i}\mu_{1,i}}{\lambda_{1,i-1}\mu_{2,i-5}}}\\
&\times\frac{1}{1-\frac{y_1\lambda_{2,1}\mu_{2,1}}{\mu_{1,1}}}\prod_{i=2}^{n}\frac{1}{1-\frac{y_i\lambda_{2,i}\mu_{2,i}}{\lambda_{2,i-1}\mu_{1,i}}}.
\end{align*}
And by computing with small values of $n$, we have the following initial values.
\begin{align*}
CP_{(5,0)}(1)=&\frac{1}{(q;q)_2},\\
CP_{(5,0)}(2)=&\frac{1}{(q;q)_4}(q^2+1),\\
CP_{(5,0)}(3)=&\frac{1}{(q;q)_6}(q^6+q^4+q^3+q^2+1),\\
CP_{(5,0)}(4)=&\frac{1}{(q;q)_{8}}(q^{12}+q^{10}+q^9+2 q^8+q^7+2 q^6+q^5+2 q^4+q^3+q^2+1).    
\end{align*}

\begin{theorem}
For $CP_{(5,0)}(n)$ with $n\geq5$ we have
\begin{align*}
CP_{(5,0)}(n)=&
\frac{1+q^{2n-5}+q^{2n-4}+q^{2n-3}+q^{2n-2}}{(1-q^{2n-1})(1-q^{2n})}CP_{(5,0)}(n-1)\\
&-\frac{q^{4n-11}(1+q+2q^2+q^3+q^4)}{(1-q^{2n-3})(1-q^{2n-2})(1-q^{2n-1})(1-q^{2n})}CP_{(5,0)}(n-2)\\
&-\frac{q^{4n-11}(1+q)}{(1-q^{2n-4})\cdots(1-q^{2n})}CP_{(5,0)}(n-3)\\
&-\frac{q^{4n-12}(1-q^{2n-6})}{(1-q^{2n-5})\cdots(1-q^{2n})}CP_{(5,0)}(n-3)\\
&+\frac{q^{6n-17}}{(1-q^{2n-5})\cdots(1-q^{2n})}CP_{(5,0)}(n-3)\\
&-\frac{q^{4n-13}}{(1-q^{2n-5})\cdots(1-q^{2n})}CP_{(5,0)}(n-4).
\end{align*}    
\end{theorem}

This implies the following recurrence on $P_{(5,0)}(n):= CP_(5,0)(n)/(q;q)_{2n}$. For $n\geq4$ we have
\begin{align*}
P_{(5,0)}(n)=&(1+q^{2n-5}+q^{2n-4}+q^{2n-3}+q^{2n-2})P_{(5,0)}(n-1)\\
&-q^{4n-11}(1+q+2q^2+q^3+q^4)P_{(5,0)}(n-2)\\
&-q^{4n-12}(1+q+q^2)P_{(5,0)}(n-3)+q^{6n-18}(1+q+q^{2}+q^{3})P_{(5,0)}(n-3)\\
&-q^{4n-13}(1-q^{2n-6})(1-q^{2n-7})P_{(5,0)}(n-4).
\end{align*}       

\begin{theorem}
For $n\geq0$, we have
    \begin{equation}\label{eq:P50formula} P_{(5,0)}(n)  = \sum_{r= -\infty}^\infty (-1)^r q^{r(7r+5)/2} {2 n\brack n-\frac{7r}{2}+5\frac{(-1)^r-1}{2}}_q.\end{equation}    
\end{theorem}
Once again the proof is automated and omitted here. Note that we have \[\lim_{n\rightarrow\infty}P_{(5,0)}(n)=\frac{1}{(q,q^6,q^7;q^7)_{\infty}},\] which is the modulo 7 Andrews-Gordon identities when $k=3$ and $i=1$.

We can prove a manifestly positive representation of \eqref{eq:P50formula} as in the previous cases too. This is the \eqref{conj1} with $k=3$ and $i=1$ and is omitted here.

\subsection{Profile $(4,1)$}
The following diagram indicates such a cylindric partition with profile $(4,1)$ and at most $n$ nonzero entries in each row.
{\small\begin{center}
$\begin{ytableau}
\none & \none & \none & \none &\none & b_{1} & b_{2} & b_{3} & b_{4} & b_{5} & b_{6} & \cdots & b_{n}\\
\none & a_{1} & a_{2} & a_{3} & a_{4} & a_{5} & a_{6} & \cdots & a_{n}\\
b_{1} & b_{2} & b_{3} & b_{4} & b_{5} & b_{6} & \cdots & b_{n}
\end{ytableau}$
\end{center}}
And we have the following crude form for the generating function.
\begin{align*}
CP_{(5,0)}(n,X,Y)=&\underset{\geq}{\Omega}\sum_{\substack{a_1,\ldots,a_n\\b_1,\ldots,b_n}}\prod_{i=1}^{n}x_i^{a_i}y_i^{b_i}\prod_{i=1}^{n}\lambda_{1,i}^{a_i-a_{i+1}}\lambda_{2,i}^{b_1-b_{i+1}}\prod_{i=1}^{n}\mu_{1,i}^{a_i-b_{i+1}}\mu_{2,i}^{b_i-a_{i+4}}\\
=&\underset{\geq}{\Omega}\frac{1}{1-x_1\lambda_{1,1}\mu_{1,1}}\prod_{i=2}^{4}\frac{1}{1-\frac{x_i\lambda_{1,i}\mu_{1,i}}{\lambda_{1,i-1}}}\prod_{i=5}^{n}\frac{1}{1-\frac{x_i\lambda_{1,i}\mu_{1,i}}{\lambda_{1,i-1}\mu_{2,i-4}}}\\
&\times\frac{1}{1-y_1\lambda_{2,1}\mu_{2,1}}\prod_{i=2}^{n}\frac{1}{1-\frac{y_i\lambda_{2,i}\mu_{2,i}}{\lambda_{2,i-1}\mu_{1,i-1}}}.
\end{align*}
Applying the Omega operator, we have the following equations for some small values of $n$.
\begin{align*}
CP_{(4,1)}(1)=&\frac{1}{(q;q)_{2}}(1+q),\\
CP_{(4,1)}(2)=&\frac{1}{(q;q)_{4}}(q^4+q^3+q^2+q+1),\\
CP_{(4,1)}(3)=&\frac{1}{(q;q)_{6}}(q^9+q^8+q^7+2 q^6+2 q^5+2q^4+2 q^3+q^2+q+1),\\
CP_{(4,1)}(4)=&\frac{1}{(q;q)_{8}}(q^{16}+q^{15}+q^{14}+2 q^{13}+3 q^{12}+3 q^{11}+4 q^{10}+4 q^9+4 q^8\\
&+4q^7+4 q^6+3 q^5+3 q^4+2q^3+q^2+q+1).    
\end{align*}
Now by a combinatorial argument similar to the proof of Theorem \ref{CP21recurrence} and Theorem \ref{CP30recurrence}, we have the following
recurrence for $CP_{(4,1)}(n)$.
\begin{theorem}
For $n\geq5$, we have
\begin{align*}
CP_{(4,1)}(n)=&
\frac{1+q^{2n-4}+q^{2n-3}+q^{2n-2}+q^{2n-1}}{(1-q^{2n-1})(1-q^{2n})}CP_{(4,1)}(n-1)\\
&-\frac{q^{4n-5}+q^{4n-6}+2q^{4n-7}+q^{4n-8}+q^{4n-9}}{(q^{2n-3};q)_{4}}CP_{(4,1)}(n-2)\\
&-\frac{q^{4n-9}(1+q)(1-q^{2n-4})+q^{4n-10}(1-q^{2n-5})}{(q^{2n-5};q)_{6}}CP_{(4,1)}(n-3)+\frac{q^{6n-14}}{(q^{2n-5};q)_6}CP_{(4,1)}(n-3)\\
&-\frac{q^{4n-11}(1-q^{2n-5})(1-q^{2n-6})}{(q^{2n-7};q)_{8}}CP_{(4,1)}(n-4).
\end{align*}
\end{theorem}
And this indicates the following recurrence for $P_{(4,1)}(n):=CP_{(4,1)}(n)/(q;q)_{2n}$. For $n\geq5$, we have
\begin{align*}
P_{(4,1)}(n)=&(1+q^{2 n-4}+q^{2n-3}+q^{2n-2}+q^{2 n-1})P_{(4,1)}(n-1)\\
&-q^{4n-9}(1+q^2)(1+q+q^2)P_{(4,1)}(n-2)\\
&-q^{4n-10}
(1+q+q^2)P_{(4,1)}(n-3)+q^{6n-15}(1+q+q^{2}+q^{3})P_{(4,1)}(n-3)\\
&-q^{4n-11}(1-q^{2n-5})(1-q^{2n-6})P_{(4,1)}(n-4).    
\end{align*}

By solving this recurrence relation and together with the initial conditions, we have the following.
\begin{theorem}
For $n\geq0$, we have
    \begin{equation}\label{eq:P41formula} P_{(4,1)}(n)  = \sum_{r= -\infty}^\infty (-1)^r q^{r(7r+3)/2} {2 n\brack n-\frac{7r}{2}+3\frac{(-1)^r-1}{2}}_q.\end{equation}    
\end{theorem}
Once again the proof is automated and omitted here. Note that we have \[\lim_{n\rightarrow\infty}P_{(4,1)}(n)=\frac{1}{(q^2,q^5,q^7;q^7)_{\infty}},\] which is the modulo 7 Andrews-Gordon identities when $k=3$ and $i=2$.

We can prove a manifestly positive representation of \eqref{eq:P41formula} as in the previous cases too. This is the \eqref{conj1} with $k=3$ and $i=2$ and is omitted here.

\subsection{Profile $(3,2)$}
The following diagram indicates such a cylindric partition with profile $(3,2)$ and at most $n$ nonzero entries in each row.
{\small\begin{center}
$
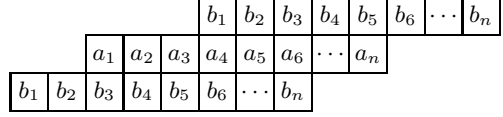
\begin{ytableau}
\none & \none & \none & \none &\none & b_{1} & b_{2} & b_{3} & b_{4} & b_{5} & b_{6} & \cdots & b_{n}\\
\none & \none & a_{1} & a_{2} & a_{3} & a_{4} & a_{5} & a_{6} & \cdots & a_{n}\\
b_{1} & b_{2} & b_{3} & b_{4} & b_{5} & b_{6} & \cdots & b_{n}
\end{ytableau}$
\end{center}}

Thus we have the following crude form for the generating function.
\begin{align*}
CP_{(5,0)}(n,X,Y)=&\underset{\geq}{\Omega}\sum_{\substack{a_1,\ldots,a_n\\b_1,\ldots,b_n}}\prod_{i=1}^{n}x_i^{a_i}y_i^{b_i}\prod_{i=1}^{n}\lambda_{1,i}^{a_i-a_{i+1}}\lambda_{2,i}^{b_1-b_{i+1}}\prod_{i=1}^{n}\mu_{1,i}^{a_i-b_{i+2}}\mu_{2,i}^{b_i-a_{i+3}}\\
=&\underset{\geq}{\Omega}\frac{1}{1-x_1\lambda_{1,1}\mu_{1,1}}\prod_{i=2}^{3}\frac{1}{1-\frac{x_i\lambda_{1,i}\mu_{1,i}}{\lambda_{1,i-1}}}\prod_{i=4}^{n}\frac{1}{1-\frac{x_i\lambda_{1,i}\mu_{1,i}}{\lambda_{1,i-1}\mu_{2,i-3}}}\\
&\times\frac{1}{1-y_1\lambda_{2,1}\mu_{2,1}}\times\frac{1}{1-\frac{\lambda_{2,2}\mu_{2,2}}{\lambda_{2,1}}}\prod_{i=3}^{n}\frac{1}{1-\frac{y_i\lambda_{2,i}\mu_{2,i}}{\lambda_{2,i-1}\mu_{1,i-2}}}.
\end{align*}
By evaluating this for some small $n$, we have the following initial conditions.
\begin{align*}
CP_{(3,2)}(1)=&\frac{1}{(q,q)_{2}}(1+q),\\
CP_{(3,2)}(2)=&\frac{1}{(q,q)_{4}}(1+q+2q^{2}+q^{3}+q^{4}),\\
CP_{(3,2)}(3)=&\frac{1}{(q,q)_{6}}(q^9+q^8+2 q^7+3 q^6+3 q^5+3 q^4+2 q^3+2 q^2+q+1),\\
CP_{(3,2)}(4)=&\frac{1}{(q,q)_{8}}(q^{16}+q^{15}+2 q^{14}+3 q^{13}+5 q^{12}+5 q^{11}+6 q^{10}+6 q^9+7 q^8\\
&+6 q^7+6 q^6+4 q^5+3 q^4+2 q^3+2 q^2+q+1).\\
CP_{(3,2)}(5)=&\frac{1}{(q;q)_{10}}(q^{25}+q^{24}+2 q^{23}+3 q^{22}+5 q^{21}+7 q^{20}+8 q^{19}+10 q^{18}+12 q^{17}\\
&+14 q^{16}+15 q^{15}+16 q^{14}+15 q^{13}+15 q^{12}+13 q^{11}+13 q^{10}+11 q^9\\
&+10 q^8+7 q^7+6 q^6+4 q^5+3 q^4+2 q^3+2 q^2+q+1)
\end{align*}

Based on the initial values, both guessing through qFunctions \cite{qFuncs} and Theorem~\ref{thm:LiUncu_AndrewsGordon_Thm} with $k=i=3$ suggests the following recurrence for what $P_{(3,2)}(n)$ would be, we denote this object by $P'_{(3,2)}(n)$.
\begin{align*}
    P'_{(3,2)}(n)=&\left(q^{2 n-4}+q^{2 n-3}+q^{2 n-2}+q^{2
   n-1}+1\right)P'_{(3,2)}(n-1)\\
   &-q^{2
   n-4}\left(q^{2 n-5}+q^{2n-4}+2 q^{2n-3}+q^{2n-2}+q^{2n-1}-q^2+1\right)P'_{(3,2)}(n-2)\\
   &-q^{4 n-8}\left(-q^{2n-7}-q^{2n-6}-q^{2n-5}-q^{2n-4}+q+2\right)P'_{(3,2)}(n-3)\\
   &-q^{4n-9}\left(1-q^{2n-7}\right)\left(1-q^{2n-6}\right)P'_{(3,2)}(n-4).
   \end{align*}
However, the combinatorial argument we applied in the previous sections won't be able to explain this. The solution for this recurrence is given as follows.
\begin{theorem}
For $n\geq0$, we have    \begin{equation}\label{eq:P32formula} P'_{(3,2)}(n)  = \sum_{r= -\infty}^\infty (-1)^r q^{r(7r+3)/2} {2 n\brack n-\frac{7r}{2}+\frac{(-1)^r-1}{2}}_q.\end{equation}    
\end{theorem}
Once again the proof is automated and omitted here. Note that we have \[\lim_{n\rightarrow\infty}P'_{(3,2)}(n)=\frac{1}{(q^3,q^4,q^7;q^7)_{\infty}},\] which is the modulo 7 Andrews-Gordon identities when $k=3$ and $i=3$. One should remember that \eqref{eq:P32formula} only provides a valid solution for the suggested recurrence. We can not use this to claim the formula of $CP_{(3,2)}(n)$, since we didn't prove the recurrence of it.

 We can prove a manifestly positive representation of \eqref{eq:P41formula} as in the previous cases too. This is the \eqref{conj1} with $k=3$ and $i=3$ and is omitted here. Also, Recall that when $k=i$, the right-hand sides of \eqref{eq:FodaQuano_FinAndrewsGordon} and \eqref{eq:LiUncu_FinAndrewsGordon} match.

\section{The Conjectural Infinite Hierarchies}\label{sec:conjs}

Studying the initial identities that finds a manifestly positive representation for the generating functions for the cylindric partition, \eqref{eq:k2i1}, \eqref{eq:k2i2}, \eqref{eq:P21Andrews}, \eqref{eq:P30Andrews}, and Theorems~\ref{thm:ManPos_P40}, \ref{thm:ManPos_P31}, and \ref{thm:ManPos_P22}, we generalized the Conjectures~\ref{conj1} and \ref{conj2}. Although we omitted proofs of Theorems~\ref{thm:LiUncu_AndrewsGordon_Thm} and \ref{thm:LU_Bressoud_k23} here, we include them in detail in the included Mathematica notebook in the ancillary files and on the second authors website \href{https://akuncu.com}{https://akuncu.com}. 

The computer algebra calculations grow with $k$. We proved all pairs $4\geq k\geq i\geq 1$ and although we could have pushed it further, we believe this is a good cut-off point. From this point on, since we have enough evidence and confidence in the conjectures, we should be looking for a general tool like a Bailey pair and appropriate Bailey lemma to prove these conjectures. 

The $k=2$ and $3$ cases are related to the MacMahon analysis we do here. Moreover,  we expect the generating function expressions to continue:
\begin{conjecture} Let $k\geq 4$ and $k\geq i\geq 1$, then
    \begin{align*}CP_{(2k-i,i-1)}(n) &= \frac{1}{(q;q)_{2n}} \sum_{r=-\infty}^{\infty} (-1)^r q^{\frac{r( (2k+1)r + 2k-2i +1)}{2}} {2n\brack n - \frac{(2k+1)r}{2} + (2k-2i+1)\frac{(-1)^r-1}{4}}_q,\\
    CP_{(2k-i-1,i-1)}(n) &=  \frac{1}{(q;q)_{2n}}\sum_{r=-\infty}^{\infty} (-1)^r q^{r( kr + k-i)} {2n\brack n -kr +(k-i) \frac{(-1)^r-1}{2}}_q.
    \end{align*}
\end{conjecture}

Following up, we also believe the nature of the sums will remain the same and the sums that appear above will always have non-negative coefficients:

\begin{conjecture}
    Let $k\geq 5$ and $k\geq i\geq 1$, then
    \[ \sum_{r=-\infty}^{\infty} (-1)^r q^{\frac{r( (2k+1)r + 2k-2i +1)}{2}} {2n\brack n - \frac{(2k+1)r}{2} + (2k-2i+1)\frac{(-1)^r-1}{4}}_q \succeq 0,\]
\[    \sum_{r=-\infty}^{\infty} (-1)^r q^{r( kr + k-i)} {2n\brack n -kr +(k-i) \frac{(-1)^r-1}{2}}_q\succeq 0,
    \]
    where $\cdot \succeq 0$ means that the object $\cdot$ has non-negative $q$-series coefficients.
\end{conjecture}

Conjectures~\ref{conj1} and \ref{conj2} are closely related to Foda--Quano's Theorems~\ref{thm:FodaQuano} and \ref{thm:FodaQuano2}, respectively. Foda--Quano proved their results through a partition interpretation of the two sides of their equation and showing combinatorially that these objects are equinumerous. For example, they showed that the left-hand side of \eqref{eq:FodaQuano_FinAndrewsGordon} is the generating function for restricted partitions which counts partitions that fit in a $\lfloor (n+1+k-i)/2 \rfloor\times \lfloor (n-k+i)/2 \rfloor$ with some successive rank considerations. They denoted this function with  $Q_{2k+1-i,i}\left( \lfloor (n+1+k-i)/2 \rfloor, \lfloor (n-k+i)/2 \rfloor; q \right)$. In this construction, the left-hand side of \eqref{eq:LiUncu_FinAndrewsGordon} would correspond to $Q_{2k+1-i,i}\left( n,n; q \right)$. 

Their bijective proof requires a refinement and keeping track of one more statistic among the counted partitions. However, this is still great insight into the combinatorial nature of these sums. Therefore, we hope that Conjectures~\ref{conj1} and \ref{conj2} can be proven. 

\section{New Infinite Hierarchies Using Bailey Machinery}\label{sec:newHierarchies}

In \cite{FodaQuano2}, Foda-Quano using the Bailey machinery following Andrews example \cite{AndrewsMultiple}, grew their identities that appear in Theorem \ref{thm:FodaQuano} into infinite series. We can follow their lead and do the same here. There are multiple Bailey pairs and lemmas one can impose. We prefer to stick with polynomial identities for the theme of this paper. To that end, we recall a theorem from Berkovich and the second author \cite[Theorem 2.1]{BerkUncu}.

\begin{theorem}(Berkovich-U)\label{thm:BU}
    For $a=0,1$, if \[F_a(L,q) = \sum_{j=-\infty}^\infty \theta_j(q) {2L +a\brack L-j}_q\] then
    \begin{equation}
        \label{eq:BU_bailey} \sum_{r\geq 0} \frac{q^{r^2+a r}(q;q)_{2L+a}}{(q;q)_{L-r}(q;q)_{2r+a}} F_a(r,q) = \sum_{j=-\infty}^\infty \theta_j(q) q^{j^2+aj}{2L+a \brack L-j}_q
    \end{equation}
\end{theorem}

First we apply \eqref{eq:BU_bailey} with $a=0$ to Theorem~\ref{thm:FodaQuano}. This yields the following theorem.

\begin{theorem}\label{thm:InfiniteFodaQuano1} Let $n,p, k,i$ be fixed integers where $n\geq 0$ and $k\geq i\geq 1$ Then, \begin{align}
\nonumber\sum_{m_p\geq \dots\geq m_{1}\geq 0} &\frac{q^{m_p^2+m_{p-1}^2+\dots +m_1^2} (q;q)_{2n}}{(q;q)_{n-m_p}(q;q)_{m_p - m_{p-1}}\dots (q;q)_{m_2-m1}(q;q)_{2m_1}} \times \\
    \label{eq:FodaQuano_FinAndrewsGordon_Bailey}\sum_{n_1\geq \dots\geq n_{k-1}\geq n_k=0} &q^{n_1^2+n_2^2+\dots+n_{k-1}^2 + n_i+\dots +n_{k-1}} \prod_{j=1}^{k-1} {2m_1- 2\sum_{l=1}^{j-1}n_l -n_j - n_{j+1}  - \alpha_{ij}\brack n_j - n_{j+1}}_q \\ \nonumber  &= \sum_{r=-\infty}^{\infty} (-1)^r q^{\frac{r( (2k+1)r + 2k-2i +1)}{2}+p \lfloor \frac{i-k-(2k+1)r}{2}\rfloor^2} {2n\brack n - \lfloor \frac{i-k-(2k+1)r}{2}\rfloor }_q,
\end{align}
where $\alpha_{ij} := \max\{j-i+1,0\}$.
\end{theorem}

The proof follows quickly by applying \eqref{eq:BU_bailey} with the initial $\theta_j(q)$, where $j:= j(r)$, as \begin{align*}\theta_j (q) &= \left\{\begin{array}{cl}  (-1)^rq^{\frac{r((2k+1)r+ (2k-2i+1))}{2}}, & \text{if }j = \lfloor \frac{i-k-(2k+1)r}{2}\rfloor \text{ for some }r\in\mathbb{Z}, \\
0, & \text{otherwise,}
\end{array} \right.\\
\intertext{or more precisely, by first splitting the right-hand sum with $r\mapsto 2r$ and $2r+1$ which yields}
&= \left\{\begin{array}{cl}  q^{2(2k+1)r^2+ (2k-2i+1)r}, & \text{if }j = (2k+1)r - \lfloor \frac{k-i}{2} \rfloor\text{ for some }r\in\mathbb{Z}, \\
-q^{2(2k+1)r^2+(6k-2i+3)r + (2k-i+1)}, & \text{if }j = (2k+1)r + \lfloor \frac{k+i+1}{2} \rfloor\text{ for some }r\in\mathbb{Z}, \\
0, & \text{otherwise.}
\end{array} \right.\end{align*}

If $k=i$, as $n\rightarrow \infty$, \eqref{eq:FodaQuano_FinAndrewsGordon_Bailey} implies the even case of \cite[Theorem 1.1]{FodaQuano2}. For $k\not= i$ cases, we get a sum-product identity that does not share the coupling of the variables $k$ and $i$ that the inner and outer multiple sums of \cite[Theorem 1.1]{FodaQuano2} possess. To be precise, they normalize their Bailey pair relative to $q^{k-i}$ and also shift $2m_1 \mapsto 2m+k-i$ which leads to having a $q^{(k-i)(m_p+\dots+m_1)}$ factor in their outer multi-sum. We instead have the following similar corollary.

\begin{corollary}
     Let $p, k,i$ be fixed integers  and $k\geq i\geq 1$ Then, \begin{align}
\nonumber\sum_{m_p\geq \dots\geq m_{1}\geq 0} &\frac{q^{m_p^2+m_{p-1}^2 +\dots+m_1^2}}{(q;q)_{m_p - m_{p-1}}\dots (q;q)_{m_2-m1}(q;q)_{2m_1}}
    \times \\
    \label{eq:FodaQuano_FinAndrewsGordon_Bailey_Corollary}\sum_{n_1\geq \dots\geq n_{k-1}\geq n_k=0} &q^{n_1^2+n_2^2+\dots+n_{k-1}^2 + n_i+\dots +n_{k-1}} \prod_{j=1}^{k-1} {2m_1- 2\sum_{l=1}^{j-1}n_l -n_j - n_{j+1}  - \alpha_{ij}\brack n_j - n_{j+1}}_q  \\ \nonumber &= \frac{1}{(q;q)_\infty}\sum_{r=-\infty}^{\infty} (-1)^r q^{\frac{r( (2k+1+2p)r + 2k-2i +1)}{2}+p \lfloor \frac{i-k-(2k+1)r}{2}\rfloor^2},
\end{align}
where $\alpha_{ij} := \max\{j-i+1,0\}$.
\end{corollary}

Note that, we can also apply Theorem~\ref{thm:BU} with $a=1$ to \cite[Proposition 3.1]{FodaQuano2} with $v\mapsto 2v+1$ in order to get hierarchies similar to \cite[Theorem 1.1]{FodaQuano2}'s odd cases.

Now we apply Theorem~\ref{thm:BU} with $a=0$ to Conjecture~\ref{conj1} to see the whole infinite hierarchy of polynomial identities seeds from it. Remark that on the right-side of \eqref{eq:BU_bailey}, we consider the summation variable $j$ as a function of $r$ and write $\theta_j(q)$ in $r$.

\begin{conjecture}\label{conj1_Bailey}
    Let $n,p, k,i$ be integers where $n\geq 0$, $k\geq 4$, and $k > i\geq 1$, then \begin{align}\nonumber&\sum_{m_p\geq \dots\geq m_{1}\geq 0} \frac{q^{m_p^2+m_{p-1}^2 +\dots+m_1^2}(q;q)_{2n}}{(q;q)_{n-m_p}(q;q)_{m_p - m_{p-1}}\dots (q;q)_{m_2-m1}(q;q)_{2m_1}}
    \times \\
    \label{eq:LiUncu_FinAndrewsGordon_Bailey}&\sum_{n_1\geq \dots\geq n_{k-1}\geq n_k=0} q^{n_1^2+n_2^2+\dots+n_{k-1}^2 + n_i+\dots +n_{k-1}} \prod_{j=1}^{k-1} {2m_1- 2\sum_{l=1}^{j-1}n_l -n_j - n_{j+1} -2 \alpha_{ij} \brack n_j - n_{j+1}}_q \\  \nonumber&\hspace{0cm}= \sum_{r=-\infty}^{\infty} (-1)^r q^{\frac{r( (2k+1)r + 2k-2i +1)}{2}+p\left(\frac{(2k+1)r}{2} - (2k-2i+1)\frac{(-1)^r-1}{4}\right)^2} {2n\brack n - \frac{(2k+1)r}{2} + (2k-2i+1)\frac{(-1)^r-1}{4}}_q,
    \end{align}
    where $\alpha_{ij}:= \max\{j-i+1,0\}$.
\end{conjecture}

\begin{theorem}
    For $k= 2, 3$ and $4$, \eqref{eq:LiUncu_FinAndrewsGordon_Bailey} holds.
\end{theorem}

The proof follows by applying Theorem~\ref{thm:BU} with $a=0$ Theorem~\eqref{thm:LiUncu_AndrewsGordon_Thm} repeatedly.

\begin{theorem}\label{thm:Rogers_Inf_Bailey}
    Let $n,p $ be integers where $n\geq 0$ and $p\geq 1$ then \begin{align}\nonumber\sum_{m_p\geq \dots\geq m_{1}\geq 0} &\frac{q^{m_p^2+m_{p-1}^2 +\dots+m_1^2}(q;q)_{2n}}{(q;q)_{n-m_p}(q;q)_{m_p - m_{p-1}}\dots (q;q)_{m_2-m1}(q;q)_{2m_1}}
    = \sum_{r=-\infty}^\infty (-1)^r q^{r(3r+1)/2+ p\lfloor\frac{3r}{2}\rfloor^2} {2n \brack n - \lfloor\frac{3r}{2} \rfloor}.
    \end{align}
\end{theorem}

This identity is acquired by applying Theorem~\ref{thm:BU} with $a=0$ to \eqref{eq:Rogers} repeatedly. The same equality can be written as \eqref{eq:Finite_Rogers_IntroThm} of Theorem~\ref{thm:Finite_Rogers_IntroThm},
which is more in the tone of \eqref{eq:LiUncu_FinAndrewsGordon_Bailey}. 

Theorem~\ref{thm:FodaQuano2} allows us to employ $a=1$ cases of Theorem~\ref{thm:BU}. However, the equation \eqref{eq:FodaQuano_FinBressoud} is limited in other ways. We can only use Theorem~\ref{thm:BU} when $k=i$ or $k=i+1$.

\begin{theorem} For integers $n\geq0$, $k\geq2$, and $p\geq0$, we have
    \begin{align}\nonumber\sum_{m_p\geq \dots\geq m_{1}\geq 0} &\frac{q^{m_p^2+m_{p-1}^2 +\dots+m_1^2}(q;q)_{2n}}{(q;q)_{n-m_p}(q;q)_{m_p - m_{p-1}}\dots (q;q)_{m_2-m1}(q;q)_{2m_1}}\times
   \\ \label{eq:FQ_Bressoud_Bailey_kEQi}\sum_{n_1\geq \dots\geq n_{k-1}\geq0} &q^{n_1^2+n_2^2+\dots+n_{k-1}^2 + n_i+\dots +n_{k-1}} {n-\sum_{j=1}^{k-2}n_j \brack n_{k-1}}_{q^2} \prod_{j=1}^{k-2} {2n- 2\sum_{l=1}^{j-1}n_l -n_j - n_{j+1}  \brack n_j - n_{j+1}}_q \\ \nonumber &= \sum_{r=-\infty}^{\infty} (-1)^r q^{ (p+1)kr^2 } {2n\brack n -kr }_q,\\
   \nonumber\sum_{m_p\geq \dots\geq m_{1}\geq 0} &\frac{q^{m_p^2+m_{p-1}^2 +\dots+m_1^2+m_p+m_{p-1} +\dots+m_1}(q;q)_{2n}}{(q;q)_{n-m_p}(q;q)_{m_p - m_{p-1}}\dots (q;q)_{m_2-m1}(q;q)_{2m_1}}\times
   \\ \label{eq:FQ_Bressoud_Bailey_kEQip1}\sum_{n_1\geq \dots\geq n_{k-1}\geq0} &q^{n_1^2+n_2^2+\dots+n_{k-1}^2 + n_i+\dots +n_{k-1}} {n-\sum_{j=1}^{k-2}n_j \brack n_{k-1}}_{q^2} \prod_{j=1}^{k-2} {2n- 2\sum_{l=1}^{j-1}n_l -n_j - n_{j+1}  + 1 \brack n_j - n_{j+1}}_q \\ \nonumber &= \sum_{r=-\infty}^{\infty} (-1)^r q^{(1+p)kr^2 + (kp +1) r} {2n+1\brack n -kr }_q.
    \end{align}
\end{theorem}

Equations \eqref{eq:FQ_Bressoud_Bailey_kEQi} and \eqref{eq:FQ_Bressoud_Bailey_kEQip1} can be proving by picking $k=i$ and $k=i+1$ in \eqref{eq:FodaQuano_FinBressoud} and applying Theorem~\ref{thm:BU} with $a=0$ and 1, respectively, repeatedly.

\begin{conjecture}     Let $n,p, k,i$ be integers where $n\geq 0$, $k\geq 4$, and $k > i\geq 1$, then
    \begin{align}
    \nonumber\sum_{m_p\geq \dots\geq m_{1}\geq 0} &\frac{q^{m_p^2+m_{p-1}^2 +\dots+m_1^2}(q;q)_{2n}}{(q;q)_{n-m_p}(q;q)_{m_p - m_{p-1}}\dots (q;q)_{m_2-m1}(q;q)_{2m_1}}\times
   \\ \label{eq:LiUncu_FinBressoud_Bailey}\sum_{n_1\geq \dots\geq n_{k-1}\geq 0} &q^{n_1^2+\dots+n_{k-1}^2 + n_i+\dots +n_{k-1}} {n-\sum_{j=1}^{k-2}n_j - k +i\brack n_{k-1}}_{q^2} \prod_{j=1}^{k-2} {2n- 2\sum_{l=1}^{j-1}n_l -n_j - n_{j+1}  - 2\alpha_{ij}\brack n_j - n_{j+1}}_q \\  \nonumber&\hspace{1cm}= \sum_{r=-\infty}^{\infty} (-1)^r q^{r( kr + k-i)+p\left(kr -(k-i) \frac{(-1)^r-1}{2}\right)^2} {2n\brack n -kr +(k-i) \frac{(-1)^r-1}{2}}_q,
\end{align}
 where $\alpha_{ij}:= \max\{j-i+1,0\}$.
\end{conjecture}

\begin{theorem}
    The equation \eqref{eq:LiUncu_FinBressoud_Bailey} holds for $k=2,3$ and $4$.
\end{theorem}

Finally, we give the infinite hierarchy of polynomial identities one can grow from \eqref{eq:LU_Bressoud_k1} using Theorem~\ref{thm:BU} with $a=0$.

\begin{theorem}
\begin{align}
    \sum_{m_p\geq \dots\geq m_{1}\geq 0} &\frac{q^{m_p^2+m_{p-1}^2 +\dots+m_1^2}(q;q^2)_{m_1}(q;q)_{2n}}{(q;q)_{n-m_p}(q;q)_{m_p - m_{p-1}}\dots (q;q)_{m_2-m1}(q;q)_{2m_1}}= \sum_{r=-\infty}^{\infty} (-1)^r q^{(p+1)r^2} {2n\brack n -r }_q
\end{align}
\end{theorem}

\section*{Acknowledgment}

The authors express gratitude to George E. Andrews, Alexander Berkovich, Ka\u{g}an Kurşungöz, and Ae Ja Yee for their encouragement and comments on the manuscript.

The second author acknowledges the support of the Austrian Science Fund FWF, P34501N, and UKRI EPSRC EP/T015713/1 projects.



\begin{thebibliography}{99}

\bibitem{qFuncs} J. Ablinger and A.K. Uncu, \texttt{qFunctions }\textit{- A Mathematica package for $q$-series and partition theory applications}, J. Symbolic Comput. 107, (2021), 145-166.

\bibitem{A} G.E, Andrews, {\it The Theory of Partitions}, Cambridge University Press (1984).

\bibitem{Andrews1} G.E, Andrews, {\it A polynomial identity which implies the Rogers-Ramanujan identities}, Scripta Mathematica, Volume XXVIII, No. 4.

\bibitem{Andrews2}
G.E. Andrews, {\it MacMahon's partition analysis I: The lecture hall partition theorem},  In: Sagan, B.E., Stanley, R.P. (eds) Mathematical Essays in honor of Gian-Carlo Rota. Progress in Mathematics, vol 161. Birkhäuser Boston. https://doi.org/10.1007/978-1-4612-4108-9\_1

\bibitem{AndrewsMultiple} G. E. Andrews, {\it Multiple series Rogers-Ramanujan type identities}. Pacific journal of mathematics 114.2 (1984): 267-283.

\bibitem{Andrews8} G.E. Andrews, P. Paule, and A. Riese,  {\it MacMahon's Partition Analysis: VIII. Plane Partition Diamonds}. Advances in Applied Mathematics, 27(2-3), (2001) pp.231-242.

\bibitem{Andrews13} G.E. Andrews, and P. Paule. {\it MacMahon's partition analysis XIII: Schmidt type partitions and modular forms}. Journal of Number Theory 234 (2022): 95-119.

\bibitem{OmegaPackage} G.E. Andrews, P. Paule, and A. Riese, \textit{MacMahon’s Partition Analysis III: The Omega Package}, European J. Combin., 22 (2001), 887-904.

\bibitem{Zaf3} T. Ayyıldız, D. N. Demirel, I. Tapan, and Z. Zafeirakopoulos. \textit{A Julia Package for Polyhedral Omega and Applications}, ACM Commun. Comput. Algebra 58, 2 (June 2024), 39–42. (2025) https://doi.org/10.1145/3712023.3712029 

\bibitem{BerkovichEvenModuli} A. Berkovich, \textit{Bressoud's identities for even moduli. New companions and related positivity results}, Discrete Mathematics 345.12 (2022): 113104.

\bibitem{BerkUncu} A. Berkovich, A. K. Uncu, {\it New infinite hierarchies of polynomial identities related to the Capparelli partition theorems}. Journal of Mathematical Analysis and Applications 506.2 (2022): 125678.

\bibitem{BreuerZaf} F. Breuer and Z. Zafeirakopoulos, \textit{Polyhedral Omega: A New Algorithm for Solving Linear Diophantine Systems} Annals of Combinatorics 21 (2017), 211-280.

\bibitem{PolyhedralOmega} F. Breuer and Z. Zafeirakopoulos, \texttt{Polyhedral Omega} \text{in Sage}, \href{https://github.com/fbreuer/polyhedral-omega-sage} Accessed: Dec 3, 2025.

\bibitem{Bressoud} D. M. Bressoud, \textit{An analytic generalization of Rogers-Ramanujan identities with interpretation}, Quart. J. Math. Oxford (2) 31 (1980), 389–399.

\bibitem{BU} W. Bridges, and A. K. Uncu, {\it Weighted cylindric partitions}. Journal of Algebraic Combinatorics 56.4 (2022): 1309-1337.

\bibitem{Borodin} A. Borodin, {\it Periodic Schur processes and cylindric partitions}, Duke Math. J. {\bf 140} (2007), 391--468.

\bibitem{SC} S. Corteel, {\it Rogers-Ramanujan identities and the Robinson-Schensted-Knuth correspondence}, Proceeding Amer. Math. Soc. Vol. 145, No. 5, (2007), 2011-2022.

\bibitem{CDU} S. Corteel, J. Dousse, and A. K. Uncu. {\it Cylindric partitions and some new $A_2$ Rogers–Ramanujan identities}. Proceedings of the American Mathematical Society 150.2 (2022): 481-497.

\bibitem{CW} S. Corteel, Sylvie, and T. Welsh, {\it The A 2 Rogers–Ramanujan identities revisited}. Annals of Combinatorics 23.3-4 (2019): 683-694.

\bibitem{FodaQuano} O. Foda, and Y.-H. Quano, {\it Polynomial identities of the Rogers-Ramanujan type}. International Journal of Modern Physics A 10.16 (1995): 2291-2315.

\bibitem{FodaQuano2} O. Foda, and Y.-H. Quano, {\it Virasoro character identities from the Andrews–Bailey construction}. International Journal of Modern Physics A 12.09 (1997): 1651-1675.

\bibitem{GesselKrattenthaler} I. Gessel, and C. Krattenthaler, {\it Cylindric partitions}. Transactions of the American Mathematical Society 349.2 (1997): 429-479.

\bibitem{qFact} A. Jiménez-Pastor, and A. K. Uncu. \textit{Factorial Basis Method for q-Series Applications}, Proceedings of the 2024 International Symposium on Symbolic and Algebraic Computation. pp. 382-390. (2024).

\bibitem{Halime2} K. Kurşungöz, and H. Ömrüuzun Seyrek. {\it A Decomposition of Cylindric Partitions and Cylindric Partitions into Distinct Parts}. arXiv preprint arXiv:2308.14514 (2023).

\bibitem{Halime1} K. Kurşungöz, and H. Ömrüuzun Seyrek, {\it Combinatorial constructions of generating functions of cylindric partitions with small profiles into unrestricted or distinct parts}. arXiv preprint arXiv:2205.13581 (2022).

\bibitem{MacMahon}
P. A. MacMahon, {\it Combinatorial analysis, Volumes I and II}, Vol. 137. American Mathematical Society, 2001.


\bibitem{KR} K. Shashank, and M. C. Russell, {\it Completing the $A_2$ Andrews–Schilling–Warnaar identities}. International Mathematics Research Notices 2023.20 (2023): 17100-17155.

\bibitem{Shunsuke} S. Tsuchioka,  {\it An example of $ A_2 $ Rogers-Ramanujan bipartition identities of level 3.} arXiv preprint arXiv:2205.04811 (2022).

\bibitem{OmegaWeb} A. Riese, \texttt{Omega} \textit{- A Mathematica Implementation of Partition Analysis} \href{https://risc.jku.at/sw/omega/} Accessed: Dec 3, 2025.

\bibitem{qMultiSum} A. Riese, \texttt{qMultiSum} \textit{- A Package for Proving q-Hypergeometric Multiple Summation Identities}, Journal of Symbolic Computation 35 (2003), 349-376.

\bibitem{Rogers} L. J. Rogers, \textit{Second memoir on the expansion of certain infinite products}, Proc. London Math. Soc, 25 (1894), 318-343.

\bibitem{U} A.K. Uncu, {\it Proofs of Modulo 11 and 13 Cylindric Kanade-Russell Conjectures for $ A_2 $ Rogers-Ramanujan Type Identities}. arXiv preprint arXiv:2301.01359 (2023).

\bibitem{W} O. Warnaar, {\it The $A_2$ Andrews–Gordon identities and cylindric partitions}. Transactions of the American Mathematical Society, Series B 10.22 (2023): 715-765.

\end{thebibliography}
\end{document}